\documentclass[11pt, twoside]{article}
\pdfoutput=1

\usepackage[margin=1in]{geometry}

\usepackage{graphicx}
\usepackage{amsmath,amsfonts,amssymb}
\usepackage[mathscr]{eucal}
\usepackage{bm}

\usepackage{amsthm}
\usepackage{cite}
\usepackage{hyperref}

\usepackage{multirow}
\usepackage{makecell}
\usepackage{booktabs}
\newcommand{\ra}[1]{\renewcommand{\arraystretch}{#1}}

\usepackage{color}                              % For creating coloured text and background

\usepackage[nameinlink]{cleveref}
\numberwithin{equation}{section}

\newtheorem{theorem}{Theorem}[section]
\newtheorem{lemma}[theorem]{Lemma}
\newtheorem{proposition}[theorem]{Proposition}

\theoremstyle{definition}

\theoremstyle{remark}

\usepackage{fancyhdr}
\begin{document}

\title {Superconvergent {I}nterpolatory  {HDG} methods for reaction diffusion equations {II}: {HHO}-inspired methods}
\author{Gang Chen%\fnref{myfootnote}}
	\thanks{School of Mathematics, Sichuan University, Chengdu, China, cglwdm@scu.edu.cn}
	\and
	Bernardo Cockburn\thanks {School of Mathematics, University of Minnesota, Minneapolis, MN, cockburn@math.umn.edu}
	\and
	John~R.~Singler\thanks {Department of Mathematics and Statistics, Missouri University of Science and Technology, Rolla, MO, singlerj@mst.edu}
	\and
	Yangwen Zhang\thanks {Department of Mathematical Science, University of Delaware, Newark, DE, ywzhangf@udel.edu}
}

\date{}
\maketitle

	\begin{abstract}
	In {[J. Sci. Comput., 81:2188-2212, 2019], we
    considered a superconvergent hybridizable discontinuous Galerkin (HDG) method, defined on simplicial meshes, for scalar reaction-diffusion equations and showed how to define 
    an interpolatory version which maintained its convergence properties. The interpolatory approach
    uses a locally postprocessed approximate solution to evaluate the nonlinear term, and assembles all HDG matrices once before the time intergration leading to a reduction in computational cost. The resulting method
    displays a superconvergent rate for the solution for polynomial degree $k\ge 1$.}  In this work, we {take advantage of the link found between the HDG and the hybrid high-order (HHO) methods, in }[ESAIM Math.\ Model.\ Numer.\ Anal., 50(3):635–650, 2016] {and extend this idea to  the new, HHO-inspired HDG methods, defined on meshes made of general polyhedral elements, uncovered therein}. We prove that the resulting interpolatory HDG methods converge at the same rate as for the linear elliptic problems. Hence, we obtain superconvergent methods for $k\ge 0$ by some methods. We  present numerical results to illustrate the convergence theory.
	\end{abstract}
	
	\textbf{Keywords} {Hybrid high order methods, hybridizable discontinuous Galerkin methods}, interpolatory method, superconvergence.

	\section{Introduction}

	{This is the third in a series of papers devoted to
	the devising of interpolatory HDG methods for the
    scalar reaction-diffusion model problem
	\begin{equation}\label{semilinear_pde1}
	\begin{split}
	\partial_tu-\Delta u+ F(u)&= f \quad  \mbox{in} \; \Omega\times(0,T],\\
	u&=0 \quad  \mbox{on} \; \partial\Omega\times(0,T],\\
	u(\cdot,0)&=u_0~~~\mbox{in}\ \Omega,
	\end{split}
	\end{equation}
	where $\Omega$ is a Lipschitz polyhedral domain in 
	$\mathbb R^d $, $ d\ge 2$, with boundary $\Gamma = \partial\Omega$.
	The interpolatory approach has two main advantageous features.  First, it avoids the use of a numerical quadrature typically required for the assembly of the global matrix at each iteration in each time step, which is a computationally costly component  of  standard HDG methods for nonlinear equations. Second, the interpolated nonlinear term and its Jacobian are simple to formulate and evaluate, which yields a straightforward implementation of the method.
	
	In the first paper of this series, \cite{CockburnSinglerZhang1}, we applied this idea 
	to an HDG method defined on simplicial meshes. It is {called the} HDG$_k$ {method} since it uses polynomials of  degree $k$ to approximate all variables, that is, the flux $\bm q = -\nabla u$, the solution $u$, and its numerical trace on the faces of the elements. The interpolatory method was obtained by simply replacing the nonlinearity $F(u_h)$ by a suitably defined linear interpolate $\mathcal I_h F(u_h)$. 
	Unfortunately, the resulting method lost the superconvergence property it had in the linear case.
	In the second paper, \cite{ChenCockburnSinglerZhang1},
	we showed that, if instead of $\mathcal I_h F(u_h)$,  we use  $\mathcal I_h F(u_h^\star)$, where
	$u_h^\star$ was a elementwise postprocessing of the approximate solution, we recovered the superconvergence previously lost. In this paper, we extend this idea to 
	the new, HHO-inspired HDG methods uncovered in  \cite{MR3507267}. These methods are defined on meshes made of general polyhedral elements, use polynomials of degree $k$ to approximate the flux variable $\bm q $ and numerical trace, and use different polynomial degrees for the scalar variable $u$. We refer to them as the HDG (ABC) methods. To deal with non-simplicial elements, the stabilization function incorporates the  postprocessed  approximation $u_h^\star$, which is the distinctive feature of the HHO methods, see \cite{MR3507267}. We prove that the interpolatory technique maintains the convergence rates of the HDG (ABC) methods while retaining all the advantages of the interpolatory approach.
	
	We note that the HDG (A) method is also known as the Lehrenfeld-Sch{\"o}berl HDG method or the HDG+ method.  This HDG method has been investigated in a number of works; see, e.g., \cite{Oikawa15,Oikawa16} and the recent papers \cite{ChenCui20,DuSayas20,YuChenPiZhang20} and the references therein.
	
	We summarize the convergence rates of the Interpolatory HDG (ABC) methods in \Cref{convergencetable}. We see that, in terms of the approximation for $u$, the Interpolatory HDG (A) method converges optimally for all $k\ge0$, the Interpolatory HDG (B) superconverges for all $k\ge 0$, and the Interpolatory HDG (C) superconverges for $k\ge2$. We must emphasize that the superconvergence of the HDG (B) methods for $k=0$ is fundamentally different from the superconvergence of the HDG$_k$ methods considered in
	\cite{ChenCockburnSinglerZhang1}, where $k\ge 1$ is required for superconvergence. This reflects the essential different nature of the HDG (ABC) and the HDG$_k$ methods. It is worthwhile to mention that the convergence rate of HDG (C) method stated in \cite{MR3507267} has an error when $k=1$ since, in the linear case, its superconvergence property is only valid for $k\ge 2$.
	This property is similar to that of the BDM mixed methods \cite{BrezziDouglasMarini85} which use the same local spaces for $\bm{q}$ and $u$.}

	\begin{table}
		\centering
		\ra{2.5}
		\begin{tabular}{@{}cccccc@{}}\toprule
			Interpolatory HDG & $ W_h$ & $u_h^\star $ & $\tau$ & flux $\bm q$ & scalar $u$\\
			\bottomrule
			(A) & $\mathcal P^{k+1}$ & $u_h$ & $1/h_K$ & $k+1$ & $k+2$ $ (k\geq 0)$\\
			(B) &$\mathcal P^{k}$ & $\mathfrak{p}_h^{k+1}(u_h,\widehat u_h)$ & $1/h_K$ & $k+1$ & $k+2$ $ (k\geq 0)$\\
			(C)  & $\mathcal P^{k-1}$ & $\mathfrak{p}_h^{k+1}(u_h,\widehat u_h)$ & $1/h_K$ & $k+1$ & $k+2$ $(k\geq 2)$\\
			\bottomrule
		\end{tabular}
		\caption{\label{convergencetable}Convergence rates for the Interpolatory HDG (ABC) methods. {The elementwise postprocessing $\mathfrak{p}_h^{k+1}(u_h,\widehat u_h)$
		is taken from \cite{MR3507267}. The last column gives the orders of convergence of $u_h^\star$ to $u$.}}
	\end{table}
	
	{
	Interpolatory finite element approaches for nonlinear partial differential equations have been investigated for many decades  because of their  computational advantages. There are many different names for these methods,  including finite element methods with interpolated coefficients, product approximation, and the group finite element method.  For more information, see \cite{MR0502033,MR641309,MR798845,MR702221,MR967844,MR731213,MR2752869,MR2273051,MR2112661, MR1030644,MR1172090,MR973559,MR3178584,MR2294957,MR1068202,MR2391691,MR3403707,MR2587427} and the references therein. Our interest in applying the interpolatory approach to the HDG methods is that, after its introduction \cite{MR2485455}
	in the framework of linear steady-state diffusion problems, they have been extended in the last decade to a wide variety of partial differential equations including
	%\cite{MR2485455,MR2772094,MR2629996,MR2513831}
	nonlinear equations like those of convection-diffusion \cite{MR3463051},
	of the $p-$Laplacian \cite{MR3463051},
	of the incompressible Navier-Stokes
	flow \cite{NguyenPeraireCockburnHDGAIAAINS10,NguyenPeraireCockburnHDGNS11,MR3626531},
	of the compressible Navier-Stokes flow \cite{PeraireNguyenCockburnHDGAIAACNS10,SchutzMayCNS13},
	of fluid dynamics
	\cite{NguyenPeraireCockburnEDG15},
	of continuum solid mechanics \cite{NguyenPeraireCM12},  of scalar hyperbolic conservation laws \cite{MoroNguyenPeraireSCL12,HuertaCasoniPeraire12,CasoniPeraireHuerta13},
	of large deformation elasticity \cite{KabariaLewCockburn15,MR2558780,TerranaNguyenBonetPeraire19}; see also the reviews \cite{CockburnDurham16,CockburnEncy18}.
	The popularity of these methods stems from the fact that
	they are discontinuous Galerkin (DG) methods amenable to static condensation. Therefore, the number of globally-coupled degrees of freedom for HDG methods is significantly lower than for standard DG methods. The application of the interpolatory approach to HDG methods for nonlinear problems render the resulting methods even more efficient to implement.} 
	
	The paper is organized as follows. We discuss the  Interpolatory HDG (ABC) formulations in  \Cref{sec:HDG}. We then analyze the semidiscrete Interpolatory HDG (ABC) methods in \Cref{Error_analysis}.  Finally, we illustrate the performance of the Interpolatory HDG (ABC) methods in  \Cref{sec:numerics} with numerical experiments. {We end with some concluding remarks.}

	%\begin{table*}\centering
	%	\ra{1.3}
	%	\begin{tabular}{@{}rrrrcrrrcrrr@{}}\toprule
	%		Interpolatory HDG && $ W_h$ && $u_h^\star $ && $\tau$ && flux && potential\\
	%		\bottomrule
	%		(A) && $\mathcal P^{k+1}$ && $u_h$ && $1/h_T$ && $k+1$ && $k+2(k\geq 0)$\\
	%		%
	%		(B) &&$\mathcal P^{k}$ && $\mathfrak{p}_h^{k+1}(u_h,\widehat u_h)$ && $1/h_T$ && $k+1$ && $k+2(k\geq 0)$\\
	%		%
	%		(C)  && $\mathcal P^{k-1}$ && $\mathfrak{p}_h^{k+1}(u_h,\widehat u_h)$ && $1/h_T$ && $k+1$ && $k+2(k\geq 2)$\\
	%		\bottomrule
	%	\end{tabular}
	%	\caption{The }
	%\end{table*}

	\section{Main results}
	\label{sec:HDG}
	In this section, we introduce the notation, define the
	interpolatory HDG  (ABC) methods, and state and briefly discuss their a priori error estimates.
	
	\subsection{Notation} To describe the Interpolatory HDG (ABC) methods, we first introduce the notation used in \cite{MR2485455}. 
	
	\subsubsection{Meshes and inner products}
	Let $\mathcal{T}_h$ be a collection of disjoint elements $K$ that partition $\Omega$.  Set $\partial \mathcal{T}_h$ to be  $\{\partial K: K\in \mathcal{T}_h\}$. For an element $K$ in  $\mathcal{T}_h$, let $e = \partial K \cap \Gamma$ denote the boundary face of $ K $ if the $d-1$ Lebesgue measure of $e$ is non-zero. For two elements $K^+$ and $K^-$ of the collection $\mathcal{T}_h$, let $e = \partial K^+ \cap \partial K^-$ denote the interior face between $K^+$ and $K^-$ if the $d-1$ Lebesgue measure of $e$ is non-zero. Let $\mathcal{E}_h^o$ and $\mathcal{E}_h^{\partial}$ denote the sets of interior and boundary faces, respectively, and let $\mathcal{E}_h$ denote the union of  $\mathcal{E}_h^o$ and $\mathcal{E}_h^{\partial}$. 
	
	For  $D\subset\mathbb{R}^d$, let $(\cdot,\cdot)_D$  denote  the $L^2(D)$ inner product and, when $\Gamma$ is the union of subsets of  $\mathbb{R}^{d-1}$, let $\langle \cdot, \cdot\rangle_{\Gamma} $ denote the $L^2(\Gamma)$ inner product. We finally set
	\begin{align*}
	(w,v)_{\mathcal{T}_h} := \sum_{K\in\mathcal{T}_h} (w,v)_K,   \quad\quad\quad\quad\left\langle \zeta,\rho\right\rangle_{\partial\mathcal{T}_h} := \sum_{K\in\mathcal{T}_h} \left\langle \zeta,\rho\right\rangle_{\partial K}.
	\end{align*}

	\subsubsection{Spaces} 
	Set%
	%\begin{subequations}
	\begin{alignat*}{4}
	\bm{V}_h  &:= \{\bm{v}\in [L^2(\Omega)]^d:&&\; \bm{v}|_{K}\in [\mathcal{P}^k(K)]^d, &&\;\forall \; K\in \mathcal{T}_h\},\\
	{W}_h  &:= \{{w}\in L^2(\Omega):&&\; {w}|_{K}\in \mathcal{P}^{\ell }(K),&&\; \forall \; K\in \mathcal{T}_h\},\\
	{Z}_h  &:= \{{z}\in L^2(\Omega):&&\; {z}|_{K}\in \mathcal{P}^{k+1}(K),&&\; \forall \; K\in \mathcal{T}_h\},\\
	{M}_h  &:= \{{\mu}\in L^2(\mathcal{\mathcal{E}}_h):&&\; {\mu}|_{e}\in \mathcal{P}^k(e), &&\;\forall \; e\in \mathcal{E}_h, &&\;\mu|_{\mathcal{E}_h^\partial} = 0\},
	\end{alignat*}
	%\end{subequations}
	where $\mathcal{P}^k(D)$ denotes the set of polynomials of degree at most $k$ on a domain $D$. In what follows, we take $\ell=k+1$, $k$, and $k-1$ to define the Interpolatory HDG (A), (B), and (C) methods,  respectively. Note that the Interpolatory HDG (C) method is only defined for $k\ge 1$.
	
	\subsubsection{Interpolators and projections}
As in \cite{ChenCockburnSinglerZhang1}, we denote by  $ \mathcal I_h $ the element-wise Lagrange  interpolation operator with respect to the finite element nodes for the space $ Z_h $.  Thus, $ \mathcal{I}_h g \in Z_h $ for any function $ g $ that is continuous on each element.

     We denote by $ \Pi_\ell^o $ ($ \ell \geq 0 $) and $ \Pi_k^\partial $ ($k\ge0$) the particular $L^2$-orthogonal projections $\Pi_{\ell}^o: L^2(K)\to \mathcal P^{\ell}(K)$ and $\Pi_k^\partial : L^2(e)\to \mathcal P^{k}(e)$, respectively, that is, 
	\begin{subequations}
		\begin{alignat}{2}
		(\Pi_{\ell}^o u, v_h)_K &= (u,v_h)_K,\quad &&\;\forall \; v_h\in \mathcal P^{\ell}(K),\label{L2_do}\\
		\langle \Pi_k^\partial  u, \widehat v_h \rangle_e &= \langle u, \widehat v_h \rangle_e,\quad &&\;\forall \;  \widehat v_h\in \mathcal P^{k}(e).\label{L2_edge}
		\end{alignat}
	\end{subequations}
	We now define an auxiliary projection related to the 
	postprocessings originally developed in \cite{GastaldiNochetto89,Stenberg88,Stenberg91,CockburnGopalakrishnanSayas10} but more intimately linked with that of the HHO methods
	\cite{DiPietroErnCRAS15,DiPietroErnLemaire14,DiPietroErn15}, see also \cite{MR3507267}.
	On an element $K\in \mathcal{T}_h$, we define the auxiliary projection  $\Pi_{k+1}^\star$ as
	\begin{align}\label{pixin}
	\Pi_{k+1}^\star u=\mathfrak{p}_h^{k+1}(\Pi^{o}_{\ell} u,\Pi^{\partial}_k u),
	\end{align}
	where $\mathfrak{p}_h^{k+1}(u_h,\widehat{u}_h)$ is the element of $\mathcal P^{k+1}(K)$ satisfying
	\begin{subequations}\label{post}
		\begin{alignat}{2}
		(\nabla\mathfrak{p}_h^{k+1}(u_h,\widehat{u}_h),\nabla z_h)_K&=-(u_h,\Delta z_h)_K+\langle\widehat{u}_h,\bm n\cdot\nabla z_h \rangle_{\partial K}&&\quad\forall \; z_h \in [\mathcal P^{k+1}_\ell(K)]^{\perp},\label{pp1}\\
		(\mathfrak{p}_h^{k+1}(u_h,\widehat{u}_h),w_h)_K&=(u_h,w_h)_K&&\quad\forall \; w_h\in \mathcal{P}^{\ell}(K),\label{pp2}
		\end{alignat}
	\end{subequations}
	 where   
%	\begin{align*}
	$[\mathcal P^{k+1}_\ell(K)]^{\perp}:=\{v_h\in \mathcal P^{k+1}(K) : (v_h,w_h)_K=0,\forall \; w_h\in\mathcal P^{\ell}(K)\}.$

	\subsection{The Interpolatory  HDG (ABC) methods}
	We can now define the Interpolatory HDG (ABC) methods as follows:
	for all $(\bm r_h,v_h,\widehat{v}_h)\in \bm V_h\times W_h\times M_h$, find
	$(\bm q_h,u_h,\widehat{u}_h)\in \bm V_h\times W_h\times M_h$ such that
	\begin{subequations}
	\label{HDG-O}
		\begin{align}
		(\bm{q}_h,\bm{r}_h)_{\mathcal{T}_h}-(u_h,\nabla\cdot \bm{r}_h)_{\mathcal{T}_h}+\left\langle\widehat{u}_h,\bm r_h\cdot\bm n \right\rangle_{\partial{\mathcal{T}_h}} &= 0, \label{HDG-O_a}\\
		(\partial_t u_h, v_h)_{\mathcal T_h}-(\bm{q}_h,\nabla v_h)_{\mathcal{T}_h}+\left\langle\widehat{\bm{q}}_h\cdot \bm{n},v_h\right\rangle_{\partial{\mathcal{T}_h}} +  ( \mathcal I_h F(u_h^\star),v_h)_{\mathcal{T}_h}&= (f,v_h)_{\mathcal{T}_h},\label{HDG-O_b}\\
		\left\langle\widehat{\bm{q}}_h\cdot \bm{n}, \widehat{v}_h\right\rangle_{\partial{\mathcal{T}_h}\backslash\mathcal{E}^{\partial}_h} &=0,\label{HDG-O_c}\\
		u_h(0)&= \overline{u}_h(0),
		\end{align}
	\end{subequations}
	where $u_h^\star  := \mathfrak{p}_h^{k+1}(u_h,\widehat{u}_h)$.
	To complete the definition of the methods, we need to define the numerical trace for the flux, $\widehat{\bm q}_h$, and the initial condition
	$\overline{u}_h(0)$.
%	\end{align*}
	%{\color{red}We note that any initial condition can be used for the algorithm, but we consider the specific choice $ \overline u_h(0)$ to simplify the analysis. How about this: We note that the $L^2$ projection of the initial condition $u_0$ can be used for the algorithm, but we consider the specific choice $ \overline u_h(0)$ to simplify the analysis.}
		For any element $K\in \mathcal{T}_h$, we define $\widehat{\bm q}_h\cdot\bm n$
	on $\partial K$ by
	\begin{subequations}
	\begin{align}\label{num_tra}
	\widehat{\bm q}_h\cdot\bm n &= \bm q_h\cdot\bm n+r_{\partial K}^{k*}[h_{K}^{-1} r_{\partial K}^{k}( u_h -\widehat u_h)],
	\end{align}
	where $r_{\partial K}^{k*}$ is the adjoint of $r_{\partial K}^{k}$, and $r_{\partial K}^{k}$ is
	defined, see \cite{MR3507267}, by
	\begin{eqnarray}
	r_{\partial K}^{k}(u_h-\widehat{u}_h)=\Pi^{\partial}_ku_h^\star-\widehat{u}_h. \label{rk}
	\end{eqnarray}
	\end{subequations}
	Finally, we define the initial condition $\overline{u}_h(0)$ as one of the components of the HDG (ABC) elliptic approximation. For any $t\in [0,T]$, we
	define the HDG (ABC) elliptic approximation of
	$(-\nabla u(t)|_{\mathcal{T}_h}, u(t)|_{\mathcal{T}_h}, u(t)|_{\mathcal{E}_h})$
	to be the unique element $(\overline{\bm q}_h,\overline{u}_h,\widehat{\overline u}_h)$ of $ \bm V_h\times W_h\times M_h$ 
	which solves
	\begin{subequations}\label{elliptic-projection}
		\begin{align}
		(\overline{\bm{q}}_h,\bm{r}_h)_{\mathcal{T}_h}-(\overline{u}_h,\nabla\cdot \bm{r}_h)_{\mathcal{T}_h}+\langle\widehat{\overline{u}}_h,\bm r_h\cdot\bm n \rangle_{\partial{\mathcal{T}_h}\setminus\partial \Omega} &= 0, \label{pi1}\\
		(\nabla\cdot\overline{\bm{q}}_h, v_h)_{\mathcal{T}_h}
		-\langle \overline{\bm q}_h\cdot\bm n,\widehat{v}_h \rangle_{\partial{\mathcal{T}_h}}+\langle
		h^{-1}_K( \Pi^{\partial}_k \overline{u}_h^\star -\widehat{\overline{u}}_h),\Pi^{\partial}_kv_h^\star-\widehat{v}_h\rangle_{\partial{\mathcal{T}_h}} &= (-\Delta u(t), v_h)_{\mathcal{T}_h}\label{pi2}
		\end{align}
	\end{subequations}
	for all $(\bm r_h,v_h,\widehat{v}_h)\in \bm V_h\times W_h\times M_h$,
	where  $ \overline{u}_h^\star  = \mathfrak{p}_h^{k+1}(\overline{u}_h,\widehat{\overline{u}}_h)$ and $
	v_h^\star=\mathfrak{p}_h^{k+1}( v_h,\widehat{v}_h)$.
	
	%The operator 
	%where $ u_h^\star  = \mathfrak{p}_h^{k+1}(u_h,\widehat{u}_h)$ is the postprocess of $u_h$ and $\widehat u_h$,  find $ \mathfrak{p}_h^{k+1}(u_h,\widehat{u}_h)\in \mathcal P^{k+1}(K)$ satisfy

	\subsection{Main result}	
	We assume that the nonlinearity $F$ satisfies a Lipschitz condition:
	\begin{align}
	|F(u) - F(v)|\le L\, |u-v|
	\quad\forall\; u,v\in D.
	\label{g_F}
	\end{align}
	As done in \cite{ChenCockburnSinglerZhang1}, we assume that, when $F$ is globally Lischitz in a suitably chosen domain $D$, the solutions of the model problem \eqref{semilinear_pde1}, and  those of the semidiscrete Interpolatory HDG (ABC) equations \eqref{HDG-O}, exist and are unique for $ t \in [0,T] $.
	
	We also assume the elliptic regularity inequality
	\begin{subequations}
	\begin{align}
	\|\bm \Phi\|_{1} + \|\Psi\|_{2} \le C   \| g\|_{0},
	\label{regular}
	\end{align}
	 where $(\bm\Phi,\Psi)$ solves the dual problem:
	\begin{equation}\label{Dual_PDE}
	\bm{\Phi}+\nabla\Psi=0,
	\quad
	\nabla\cdot\bm \Phi =g\quad~\text{in}\ \Omega,
	\quad
	\Psi = 0\qquad\text{on}\ \partial\Omega.
		\end{equation}
	\end{subequations}
	%  \varepsilon_h^{u^\star}  !!!!
	
	We can now state our main result for the Interpolatory HDG (ABC) methods. 
	\begin{theorem}\label{main_err_qu} 
		Assume that the nonlinearity $F$ is globally Lipschitz, that is, it satisfies condition \eqref{g_F} with $D:=\mathbb{R}$. Assume that
		$u \in C^1[0,T;H^{k+2}(\Omega)]$ . {Finally, assume that the elliptic regularity inequality \eqref{regular} holds.} Then, for all $0\le t\le T$, the solution $ (\bm q_h, u_h,u_h^\star) $ of the Interpolatory HDG (ABC) equations satisfies
		\begin{alignat*}{1}
		\|\bm q(t) - \bm q_h(t)\|_{\mathcal T_h}&\le C\, h^{k+1}, %{\|u\|_{L^\infty(0,T;H^{k+2}(\Omega))}},
		\\
		\|u(t) - u_h(t)\|_{\mathcal T_h}
		&\le C h^{\ell+1},
		%{\|u\|_{L^\infty(0,T;H^{k+2}(\Omega))}},
		\\
		\|u(t) - u_h^\star(t)\|_{\mathcal T_h}
		&\le
		C\;\begin{cases}
		 h^{2\phantom{+k}}
		 %{\,\|u\|_{L^\infty(0,T;H^{3}(\Omega))}} 
		 &\text{ if }(k,l)=(1,0), \\
		 h^{k+2}
		 %{\|u\|_{L^\infty(0,T;H^{k+2}(\Omega))}}
		 &\text{ otherwise}.
		\end{cases}
		\end{alignat*}
{The constant $C$ is independent of $h$, but depends on $T$ and on norms of $u$ and $u_t$.}
		Moreover, if the nonlinearity $F$ satisfies the Lipschitz condition \eqref{g_F} with $D:=[-M,M]$, where
		\begin{align*}
	M = \max\{|u(t,x)| : x\in \overline{\Omega}, \: t\in [0,T]\} + \delta, \quad  \mbox{for a fixed $ \delta>0 $,}
	\end{align*} and the mesh is quasi-uniform and $h$ is small enough, then the same convergence rates hold.
	\end{theorem}
This result states that, provided the solution is smooth enough, we recover the optimal orders of convergence. For HDG (A) with $k\ge0$, the optimal order of convergence of $k+2$ holds for $u^\star_h$, as it coincides with $u_h$.
Superconvergence of order $k+2$ for $u_h^\star$ holds for HDG (B) with $k\ge0$, and for HDG (C) with $k\ge2$. When $k=1$, the order of convergence of
$u_h^\star$ for HDG (C) is only $k+1=2$.

The result can be extended to other initial conditions{, as confirmed by our numerical experiments}. The one we chose makes the proof simpler. 
	
	\section{Proof of the error estimates}
	\label{Error_analysis}
	This section is devoted to proving our main result, the a priori error estimates of  \Cref{main_err_qu}.
	To do that, we essentially follow the approach carried out in \cite{ChenCockburnSinglerZhang1}. However, we need to use different auxiliary projections to capture the special structure of the stabilization functions of the HDG (ABC) methods.

	\subsection{Reformulating the HDG (ABC) methods}
	\label{reformulation}
	We begin by rewriting the definition of the
	Interpolatory HDG (ABC) methods to render it more suitable to our error analysis. 
	Unlike the approach used in  \cite{ChenCockburnSinglerZhang1}, here we eliminate the numerical trace of the flux from the equations.
	\begin{proposition}[Reformulation of the methods]
	\label{lemma1:HHO_proj}
		For all $(\bm r_h,v_h,\widehat{v}_h)\in \bm V_h\times W_h\times M_h$, the Interpolatory HDG (ABC) formulations can be rewritten as follows: find $(\bm q_h,u_h,\widehat{u}_h)\in \bm V_h\times W_h\times M_h$ satisfying
		%\begin{subequations}\label{HDGABC}
			\begin{align*}
			(\bm{q}_h,\bm{r}_h)_{\mathcal{T}_h}-(u_h,\nabla\cdot \bm{r}_h)_{\mathcal{T}_h}+\left\langle\widehat{u}_h,\bm r_h\cdot\bm n \right\rangle_{\partial{\mathcal{T}_h}} &= 0, %\label{HDGABC_a}
			\\
			(\partial_t u_h,v_h)_{\mathcal T_h}+( \mathcal I_h F(u_h^\star),v_h)_{\mathcal{T}_h}+		(\nabla\cdot\bm{q}_h, v_h)_{\mathcal{T}_h}
			-\langle \bm q_h\cdot\bm n,\widehat{v}_h \rangle_{\partial{\mathcal{T}_h}} \quad&
			\nonumber\\+\langle
			h^{-1}_K( \Pi^{\partial}_k u_h^\star -\widehat u_h),\Pi^{\partial}_kv_h^\star-\widehat{v}_h\rangle_{\partial{\mathcal{T}_h}} &= (f,v_h)_{\mathcal{T}_h},%\label{HDGABC_b}
			\\
			u_h(0)&= \overline{u}_h(0),
			\end{align*}
		%\end{subequations}
		where
		$
		u_h^\star=\mathfrak{p}_h^{k+1}( u_h,\widehat{u}_h)$ and $
		v_h^\star=\mathfrak{p}_h^{k+1}( v_h,\widehat{v}_h)$.
	\end{proposition}
	\begin{proof}
		Inserting the definition of the numerical trace of the flux \eqref{num_tra} into the first two equations defining the HDG  (ABC) method \eqref{HDG-O}, we obtain
		%\begin{subequations}
			\begin{align*}
			(\partial_t u_h,v_h)_{\mathcal T_h}-(\bm{q}_h,\nabla v_h)_{\mathcal{T}_h} +  ( \mathcal I_h F(u_h^\star),v_h)_{\mathcal{T}_h} \quad&\nonumber\\
			+\langle\bm q_h\cdot\bm n+r_{\partial K}^{k*}[h_{K}^{-1}( \Pi^{\partial}_k u_h^\star -\widehat u_h)],v_h\rangle_{\partial{\mathcal{T}_h}}&= (f,v_h)_{\mathcal{T}_h},%\label{p1}
			\\
			\langle\bm q_h\cdot\bm n+r_{\partial K}^{k*}[h_{K}^{-1}( \Pi^{\partial}_k u_h^\star -\widehat u_h)], \widehat{v}_h\rangle_{\partial{\mathcal{T}_h}} &=0.%\label{p2}
			\end{align*}
		%\end{subequations}
		Subtracting the second equation from the first, and integrating by parts, we get
		\begin{align*}
		(\partial_t u_h,v_h)_{\mathcal T_h}+( \mathcal I_h F(u_h^{\star}),v_h)_{\mathcal{T}_h}+(\nabla\cdot\bm{q}_h,v_h)_{\mathcal{T}_h}-\left\langle\bm q_h\cdot\bm n, \widehat{v}_h\right\rangle_{\partial{\mathcal{T}_h}} \quad &
		\nonumber\\+\langle r_{\partial K}^{k*}[h_{K}^{-1}( \Pi^{\partial}_k u_h^\star -\widehat u_h)],v_h-\widehat{v}_h\rangle_{\partial{\mathcal{T}_h}} &= (f,v_h)_{\mathcal{T}_h}.
		\end{align*}
		Since $r_{\partial K}^{k*}$ is the adjoint of $r_{\partial K}^{k}$, the result follows after using the
		definition of $r_{\partial K}^{k}$ in \eqref{rk}, $
		r_{\partial K}^{k}(v_h-\widehat v_h)= \Pi^{\partial}_k\mathfrak{p}_h^{k+1}( v_h,\widehat{v}_h)-\widehat{v}_h$,
		and after recalling that $v_h^\star=\mathfrak{p}_h^{k+1}( v_h,\widehat{v}_h)$.
		%, we finally get that
		%\begin{align*}
		%(\partial_t u_h,v_h)_{\mathcal T_h}+( \mathcal I_h F(u_h^\star),v_h)_{\mathcal{T}_h}+(\nabla\cdot\bm{q}_h,v_h)_{\mathcal{T}_h}-\left\langle\bm q_h\cdot\bm n, \widehat{v}_h\right\rangle_{\partial{\mathcal{T}_h}} \quad&
		%\nonumber\\+\langle h_{K}^{-1}( \Pi^{\partial}_k u_h^\star -\widehat u_h),\Pi^{\partial}_k v^\star_h-\widehat{v}_h\rangle_{\partial{\mathcal{T}_h}} &= (f,v_h)_{\mathcal{T}_h}.
		%\end{align*}
		%This completes the proof.
	 \end{proof}
	
	\subsection{Main error estimate}% for different types of nonlinearity}
	\label{Errorestimationfordifferenttypesnonlinearity}
	
	Our analysis is based on estimating the following quantities:
	\begin{align*}
	e_h^{\bm q}=\bm q_h -\overline{\bm q}_h, \quad  e_h^{ u}= u_h - \overline{u}_h, \quad
	e_h^{ \widehat{u}}=\widehat{u}_h - {\widehat{\overline u}}_h,\quad  e_h^{u^{\star}}=u_h^{\star}-\overline{u}_h^{\star}.
	\end{align*}
	Here, we obtain the main estimates for these functions.
	
	We begin by obtaining the error equations.
	\begin{lemma} [Error equations] \label{error}
	We have
	%	\begin{subequations}
			\begin{align*}
			(e_h^{\bm q},\bm{r}_h)_{\mathcal{T}_h}-(e^u_h,\nabla\cdot \bm{r}_h)_{\mathcal{T}_h}+\langle e^{\widehat{u}}_h,\bm r_h\cdot\bm n \rangle_{\partial{\mathcal{T}_h}} &= 0, %\label{error1}
			\\
			(\partial_t e^u_h,v_h)_{\mathcal T_h}+		(\nabla\cdot e^{\bm{q}}_h, v_h)_{\mathcal{T}_h}
			-\langle e^{\bm q}_h\cdot\bm n,\widehat{v}_h \rangle_{\partial{\mathcal{T}_h}} \quad&
			\nonumber\\+\langle
			h^{-1}_K( \Pi^{\partial}_k e^{u^\star}_h -e^{\widehat u}_h),\Pi^{\partial}_kv_h^\star-\widehat{v}_h\rangle_{\partial{\mathcal{T}_h}} +(\mathcal I_h F(u_h^\star) - F(u),v_h)_{\mathcal T_h} &= (\partial_t(\Pi^o_{\ell}u-\overline{u}_h),v_h)_{\mathcal{T}_h}.%\label{error2}
			\end{align*}
	%	\end{subequations}
	\end{lemma}
	This result can be easily proven by  subtracting the equations  \eqref{elliptic-projection} from those in  \Cref{lemma1:HHO_proj}, and noting that $e_h^{u^{\star}}=\mathfrak{p}_h^{k+1}(e_h^{ u},  e_h^{ \widehat{u}} )$.% and that $e_h^{ u}(0)= u_h(0) - \overline{u}_h(0)=0$.

	\begin{lemma}[Error estimates at $t=0$]
	\label{error_0} We have $e^u_h(0)=0$ and
		\begin{gather*}
		\|e_h^{\bm q}(0)\|^2_{\mathcal{T}_h}+\|h_K^{-1/2}(\Pi^{\partial}_ke_h^{u^{\star}}(0) - e_h^{\widehat u}(0))\|_{\partial\mathcal T_h}^2=0.
		%\\
		%\|\Pi_{k+1}^\star u(0) - u_h^\star(0)\|_{\mathcal T_h}  \le Ch^{\ell +2+\min\{\ell,1\}} + Ch^{k+2}. \label{xin-0}
		\end{gather*}
	\end{lemma}

	\begin{proof}
		Take ${(\bm r_h,v_h,\widehat{v}_h)}:=(e_h^{\bm q}(0),e_h^u(0),e_h^{\widehat{u}}(0))$ in
		the error equations of  \Cref{error}, evaluate at $t=0$ and add the resulting equations. Since $e_h^{ u}(0)= u_h(0) - \overline{u}_h(0)=0$, we get the result.
		%\begin{align*}
		%\hspace{1em}&\hspace{-1em}\|e_h^{\bm q}(0)\|^2_{\mathcal{T}_h}+\|h_K^{-1/2}(\Pi^{\partial}_ke_h^{u^{\star}}(0) - e_h^{\widehat u}(0))\|_{\partial\mathcal T_h}^2=0.
		%\end{align*}
%		This  with Lemma \ref{super_con} proves our result.
	%{This completes the proof.  }
	\end{proof}
	
	Next, we display the main error estimates.
	%\begin{proof} We subtract \eqref{pi1} from \eqref{HDGABC_a} to get \eqref{error1}.
	%Subtracting \eqref{pi2} from \eqref{HDGABC_b} to get
	%\begin{gather*}
	%	(\partial_t e^u_h,v_h)_{\mathcal T_h}+		(\nabla\cdot e^{\bm{q}}_h, v_h)_{\mathcal{T}_h}
	%		-\langle e^{\bm q}_h\cdot\bm n,\widehat{v}_h \rangle_{\partial{\mathcal{T}_h}}+\left\langle
	%		h^{-1}_K( \Pi^{\partial}_k e^{u^\star}_h -e^{\widehat u}_h),\Pi^{\partial}_kv_h^\star-\widehat{v}_h)\right\rangle_{\partial{\mathcal{T}_h}} \\
	%		 +(\mathcal I_h F(u_h^\star) - F(u),v_h)_{\mathcal T_h} = (f-\mathcal I_h F(u_h^{\star}),v_h)_{\mathcal{T}_h}\nonumber\\
	%\quad+(\Delta u-\partial_t\overline{u}_h,v_h)_{\mathcal{T}_h}.
	%		\end{gather*}
	%Since $\partial_t$ and $\Pi_{\ell}^o$ commute and by \eqref{semilinear_pde1} we get
	%\begin{align*}
	%(f-\mathcal I_h F(u_h^{\star})+\Delta u-\partial_t\overline{u}_h,v_h)_{\mathcal{T}_h}
	%&=(F(u)-\mathcal I_h F(u_h^{\star})+\partial_t u-\partial_t\overline{u}_h,v_h)_{\mathcal{T}_h}\nonumber\\
	%&=(F(u)-\mathcal I_h F(u_h^{\star})+\partial_t\Pi_{\ell}^o u-\partial_t\overline{u}_h,v_h)_{\mathcal{T}_h}.
	%	\end{align*}
	%The desired results followed immediately.
	%{This completes the proof.  } \end{proof}
	{
	\begin{lemma} [Main error estimates]\label{error_t} For $t\in [0,T]$, we have
	\begin{alignat*}{2}
		&\|e_h^{u}(t)\|^2_{\mathcal T_h} +  \int_0^t \big(\|e_h^{\bm q}\|^2_{\mathcal{T}_h}
		+\|h_K^{-1/2}(\Pi^{\partial}_ke_h^{u^{\star}} - e_h^{\widehat u})\|_{\partial\mathcal T_h}^2\big)
		&&\le 2\,t\Theta(t),
		\\
		&\| e_h^{\bm q}\|^2_{\mathcal T_h} +
		\|
		h^{-1/2} \Pi^{\partial}_k e^{u^\star}_h -e^{\widehat u}_h\|^2_{\partial{\mathcal{T}_h}} +\int_0^t\|\partial_te_h^u\|^2_{\mathcal{T}_h}
		&& \le 2\,\Theta(t),
	\end{alignat*}
	where $\Theta(t):=\int_0^t \|\partial_t(\Pi^o_{\ell}u-\overline{u}_h)\|^2_{\mathcal{T}_h}
		+\int_0^t \|F(u)-\mathcal I_h F(u_h^\star)\|^2_{\mathcal{T}_h}$.
	\end{lemma}

	\begin{proof}
		We first take $(\bm r_h,v_h,\widehat{r}_h):=(e_h^{\bm q},e_h^u,e_h^{\widehat{u}})$ in the error equations of  \Cref{error}, and add the resulting equations to get
		\[
		(\partial_t e_h^{u}, e_h^{u})_{\mathcal T_h} + \|e_h^{\bm q}\|^2_{\mathcal{T}_h}
		+\|h_K^{-1/2}(\Pi^{\partial}_ke_h^{u^{\star}} - e_h^{\widehat u})\|_{\partial\mathcal T_h}^2=(\partial_t(\Pi^o_{\ell}u-\overline{u}_h),e_h^{u})_{\mathcal{T}_h}+(F(u)-\mathcal I_h F(u_h^\star),e_h^{u})_{\mathcal{T}_h}.
		\]
		We now apply the Cauchy-Schwarz inequality to both terms of the right-hand side and then use a Gronwall-like inequality \cite[Proposition 3.1]{ChabaudCockburn12} and the fact that $e^u_h(0)=0$ to obtain
		\begin{alignat*}{2}
				&\|e_h^{u}(t)\|^2_{\mathcal T_h} + 2 \int_0^t \big(\|e_h^{\bm q}\|^2_{\mathcal{T}_h}
		+\|h_K^{-1/2}(\Pi^{\partial}_ke_h^{u^{\star}} - e_h^{\widehat u})\|_{\partial\mathcal T_h}^2\big)
		&&\le \bigg(\int_0^t \|\partial_t(\Pi^o_{\ell}u-\overline{u}_h)\|_{\mathcal{T}_h}
		\\
		&&&
		+\int_0^t\|F(u)-\mathcal I_h F(u_h^\star)
		\|_{\mathcal{T}_h}\bigg)^2.
	\end{alignat*}
		The first inequality is obtained after simple manipulations.
		
		Next, we take the partial derivative of with respect to $t$ in the first error equation of \Cref{error}
		and keep the second equation unchanged. We obtain
		\begin{align*}
		(\partial_te_h^{\bm q},\bm{r}_h)_{\mathcal{T}_h}-(\partial_te^u_h,\nabla\cdot \bm{r}_h)_{\mathcal{T}_h}+\langle \partial_te^{\widehat{u}}_h,\bm r_h\cdot\bm n \rangle_{\partial{\mathcal{T}_h}} &= 0,\\
		(\partial_t e^u_h,v_h)_{\mathcal T_h}+		(\nabla\cdot e^{\bm{q}}_h, v_h)_{\mathcal{T}_h}
		-\langle e^{\bm q}_h\cdot\bm n,\widehat{v}_h \rangle_{\partial{\mathcal{T}_h}} \quad&
		\nonumber\\+\langle
		h^{-1}_K( \Pi^{\partial}_k e^{u^\star}_h -e^{\widehat u}_h),\Pi^{\partial}_kv_h^\star-\widehat{v}_h\rangle_{\partial{\mathcal{T}_h}}  +(\mathcal I_h F(u_h^\star) - F(u), v_h)_{\mathcal{T}_h}&= (\partial_t(\Pi^o_{\ell}u-\overline{u}_h),v_h)_{\mathcal{T}_h}.
		\end{align*}
		Taking $(\bm r_h,v_h,\widehat{r}_h):=(e_h^{\bm q},\partial_te_h^u,\partial_te_h^{\widehat{u}})$ in {these equations and adding them, we get}
		\begin{align*}%\label{norm2}
		\hspace{1em}&\hspace{-1em}(\partial_t e_h^{\bm q}, e_h^{\bm q})_{\mathcal T_h} +
		\langle
		h^{-1}_K( \Pi^{\partial}_k e^{u^\star}_h -e^{\widehat u}_h), \partial_t\Pi^{\partial}_k e^{u^\star}_h -\partial_te^{\widehat u}_h\rangle_{\partial{\mathcal{T}_h}} +\|\partial_te_h^u\|^2_{\mathcal{T}_h}\nonumber\\
		& =(\partial_t(\Pi^o_{\ell}u-\overline{u}_h),\partial_te_h^{u})_{\mathcal{T}_h}+(F(u)-\mathcal I_h F(u_h^\star),\partial_te_h^{u})_{\mathcal{T}_h}.
		\end{align*}
We now apply the Cauchy-Schwarz inequality to each of the two terms of the right-hand side, use Young's inequality and the estimates of the errors at $t=0$ of  \Cref{error_0} to get the second estimate.
%We now apply the Cauchy-Schwarz inequality to each of the two terms of the right-hand side, use Young's inequality and use a Gronwall-like inequality {and the estimates of the errors at $t=0$ of  \Cref{error_0}} to get the second estimate.  {This completes the proof.  } \end{proof}
\end{proof}

	%		\begin{lemma}\label{colorray_ustar}
%		For any $ \overline{u}_h $ and $ \overline{\bm q}_h $, we have 
%		\begin{align*}
%		\|\Pi_{k+1}^\star u - u_h^\star\|_{\mathcal T_h}  &\le C ( \|u_h-\overline{u}_h\|_{\mathcal T_h}+h\|\bm{q}_h-\overline{\bm{q}}_h\|_{\mathcal T_h} )\nonumber\\
%		&\quad + C ( \|\overline{u}_h-\Pi_{\ell}^o u\|_{\mathcal T_h}+h\|\overline{\bm{q}}_h-\bm{\Pi}^o_k\bm{q}\|_{\mathcal T_h} +h\|\bm{q}-\bm{\Pi}^o_k\bm{q}\|_{\mathcal T_h} ).
%		\end{align*}
%	\end{lemma}

	\subsection{The Lipschitz conditions on the nonlinearity}
	Here, we end our error analysis. We bound the term $\|F(u)-\mathcal I_h F(u_h^\star)\|_{\mathcal{T}_h}$ under different assumptions on the nonlinearity $F(u)$ and conclude. To do that, we need
	the following auxiliary result. Its proof is given in  \Cref{HHOProjections}.
	
			\begin{lemma}\label{super_con}
		We have
		\begin{align*}
		\|\Pi_{k+1}^\star u - u_h^\star\|_{\mathcal T_h}  \le C ( \|u_h-\Pi_{\ell}^o u\|_{\mathcal T_h}+h\|\bm{q}_h-\bm{\Pi}^o_k\bm{q}\|_{\mathcal T_h} +h\|\bm{q}-\bm{\Pi}^o_k\bm{q}\|_{\mathcal T_h} ).
		\end{align*}
	\end{lemma}
	
	\subsubsection{Error estimates for a global Lipschitz condition}
	\label{ErrorestimatesunderGlobalLipschitzcondition}
	Here, we assume the nonlinearity is globally Lipschitz.

\begin{lemma}\label{non_est}
		We have 
		\begin{align*}
		\|F( u)-\mathcal I_hF( u_h^\star)\|_{\mathcal{T}_h}\le
		\| F(u)- \mathcal{I}_h F(u)\|_{\mathcal T_h} + C\big(\|  u- \mathcal I_h u\|_{\mathcal T_h} +\|  u- \Pi_{k+1}^\star u\|_{\mathcal T_h} + \| \Pi_{k+1}^\star u- u_h^\star\|_{\mathcal T_h}\big).
		\end{align*}
	\end{lemma}
	
	\begin{proof}
	
	To bound the error in the nonlinear term, we write
	$
	F( u)-\mathcal I_hF( u_h^\star) = R_1 + R_2 + R_3$,
			where
		\begin{alignat*}{1}
		   R_1&:=F(u)- \mathcal I_h F(u),
		   \quad
		   R_2:=\mathcal I_h F(u)  -  \mathcal I_h F(\Pi_{k+1}^\star u),
		   \quad
		   R_3:=
\mathcal I_h F(\Pi_{k+1}^\star u) -\mathcal I_hF( u_h^\star).
\end{alignat*}
The result follows since
\begin{alignat*}{2}
	\|R_2 \|_{\mathcal T_h}&\le C (\|  u- \mathcal I_h u\|_{\mathcal T_h} +\|  u- \Pi_{k+1}^\star u\|_{\mathcal T_h}  )\quad\text{ and }\quad
		\| R_3\|_{\mathcal T_h}  \le C\| \Pi_{k+1}^\star u- u_h^\star\|_{\mathcal T_h}.
		\end{alignat*}
	as shown in \cite{CockburnSinglerZhang1}.
	This completes the proof.% and we omit the details.
	\end{proof}
	
	\begin{lemma}\label{estimateTheta}
	For $t\in [0,T]$, we have
	that 
	\[
	\Theta(t)\le \Theta_{HDG}(t)+\Theta_{APP}(t)+
		C\int_0^t(\|e^u_h\|^2_{\mathcal T_h}+h^2\|e^{\bm{q}}_h\|^2_{\mathcal T_h}),
	\]
	where
	\begin{alignat*}{1}
		\Theta_{HDG}(t)&:= 
		\int_0^t
		\|\partial_t(\Pi^o_{\ell}u-\overline{u}_h)\|^2_{\mathcal{T}_h}
		+C \int_0^t( \|\overline{u}_h-\Pi_{\ell}^o u\|^2_{\mathcal T_h}+h^2\|\overline{\bm{q}}_h-\bm{\Pi}^o_k\bm{q}\|^2_{\mathcal T_h}),
		\\
		\Theta_{APP}(t)&:=C
		\int_0^t (\|F(u)-\mathcal I_h F(u)\|^2		_{\mathcal{T}_h}+ \|u-\mathcal I_h u\|^2_{\mathcal{T}_h}
		+\|u-\Pi_{k+1}^\star u\|^2_{\mathcal{T}_h}+h^2\|\bm{q}-\bm{\Pi}^o_k\bm{q}\|^2_{\mathcal T_h} ).
		\end{alignat*}
	\end{lemma}
	
	We note that $ \Theta_{HDG} $ involves the HDG elliptic approximation, while $ \Theta_{APP} $ involves only approximations of the exact solution of the PDE and related quantities. 
	
	\begin{proof} We have, by \Cref{non_est},
		\begin{alignat*}{1}
		\Theta_h&:=\|F(u)-\mathcal I_h F(u_h^\star)\|_{\mathcal{T}_h}
	    \\
	    &\le \|F(u)-\mathcal I_h F(u)\|_{\mathcal{T}_h}+C( \|u-\mathcal I_h u\|_{\mathcal{T}_h} +\|u-\Pi_{k+1}^\star u\|_{\mathcal{T}_h} +\|\Pi_{k+1}^\star u-u_h^{\star}\|_{\mathcal{T}_h} )
	    \\
	    &\le \|F(u)-\mathcal I_h F(u)\|_{\mathcal{T}_h}+C( \|u-\mathcal I_h u\|_{\mathcal{T}_h} +\|u-\Pi_{k+1}^\star u\|_{\mathcal{T}_h}) 
		\\
		&\quad+C ( \|{u}_h-\Pi_{\ell}^o u\|_{\mathcal T_h}+h\|{\bm{q}}_h-\bm{\Pi}^o_k\bm{q}\|_{\mathcal T_h} +h\|\bm{q}-\bm{\Pi}^o_k\bm{q}\|_{\mathcal T_h} ),
		\end{alignat*}
		by \Cref{super_con}. Using the definition of $e^u_h$ and $e^{\bm q}_h$, and the triangle inequality, we get
			\begin{alignat*}{1}
		\|F(u)-\mathcal I_h F(u_h^\star)\|_{\mathcal{T}_h}
	    &\le \|F(u)-\mathcal I_h F(u)\|_{\mathcal{T}_h}+C( \|u-\mathcal I_h u\|_{\mathcal{T}_h} +\|u-\Pi_{k+1}^\star u\|_{\mathcal{T}_h}) 
		\\
		&\quad+C ( \|\overline{u}_h-\Pi_{\ell}^o u\|_{\mathcal T_h}+h\|\overline{\bm{q}}_h-\bm{\Pi}^o_k\bm{q}\|_{\mathcal T_h} +h\|\bm{q}-\bm{\Pi}^o_k\bm{q}\|_{\mathcal T_h} )
		\\
		&\quad+C(\|e^u_h\|_{\mathcal T_h}+h\|e^{\bm{q}}_h\|_{\mathcal T_h}).
		\end{alignat*}
		Inserting this bound in the definition of $\Theta(t)$, we obtain the desired result. This completes the proof.
		\end{proof}
	
    \begin{lemma}\label{lemma:global_Lip_error_estimate}  
    For $t\in [0,T]$, we have
	\begin{alignat*}{2}
		&\|e_h^{u}(t)\|^2_{\mathcal T_h} +  \int_0^t \big(\|e_h^{\bm q}\|^2_{\mathcal{T}_h}
		+\|h_K^{-1/2}(\Pi^{\partial}_ke_h^{u^{\star}} - e_h^{\widehat u})\|_{\partial\mathcal T_h}^2\big)
		&&\le 2\,t\Phi(T),
		\\
		&\| e_h^{\bm q}\|^2_{\mathcal T_h} +
		\|
		h^{-1/2} \Pi^{\partial}_k e^{u^\star}_h -e^{\widehat u}_h\|^2_{\partial{\mathcal{T}_h}} +\int_0^t\|\partial_te_h^u\|^2_{\mathcal{T}_h}
		&& \le 2\,\Phi(T),
	\end{alignat*}
	where $\Phi(T):=C(T)(\Theta_{HDG}(T)+\Theta_{APP}(T))$.
	\end{lemma}
	
	\begin{proof} By the previous lemma, we have, for all $t\in [0,T]$,
	\begin{alignat*}{1}
	\Theta(t)\le& \Theta_{HDG}(T)+\Theta_{APP}(T)+
		C\int_0^t(\|e^u_h\|^2_{\mathcal T_h}+h^2\|e^{\bm{q}}_h\|^2_{\mathcal T_h})
	\\
	\le& \Theta_{HDG}(T)+\Theta_{APP}(T)+
		C\int_0^t(s+h^2)\,\Theta(s)\,ds,
    \end{alignat*}
	by \Cref{error_t}. By applying the Gronwall inequality, we get that $\Theta(t)\le C(T)\,\Phi(T)$. The result now follows by using the main estimates of \Cref{error_t}.% This completes the proof.
 \end{proof}
	%\begin{corollary} For $u$ smooth enough, we have:
	%when $k=1,\ell=0$
	%\begin{align*}
	%\|u-u_h\|_{\mathcal{T}_h}\le Ch^{\ell+1},\ \|\bm q-\bm q_h\|_{\mathcal{T}_h}\le C h^{k+1},  \|u-u_h^{\star}\|_{\mathcal{T}_h}\le C h^{k+1},
	%\end{align*}
	%otherwise  we have
	%\begin{align*}
	%\|u-u_h\|_{\mathcal{T}_h}\le Ch^{\ell+1},\ \|\bm q-\bm q_h\|_{\mathcal{T}_h}\le C h^{k+1},  \|u-u_h^{\star}\|_{\mathcal{T}_h}\le C h^{k+2}.
	%\end{align*}
	%\end{corollary}

	%{\color{red}The convergence estimates for $ \bm q - \bm q_h $ and $ u - u_h $ in the main result follow from the triangle inequality and the estimates in \Cref{lemmainter} and \Cref{error_dual2}.  The superconvergent estimate for $ u - u_h^\star $ in the main result follows follow from the triangle inequality, \Cref{lemma1:HHO_proj}, \Cref{colorray_ustar}, and the estimates in \Cref{lemmainter} and \Cref{error_dual2}. How about this: Since $ \bm q - \bm q_h = \bm q - \bm \Pi_k^o \bm q + \bm \Pi_k^o \bm q  - \overline{\bm q}_h + \overline{\bm q}_h- \bm q_h$ and $u-u_h = u -  \Pi_k^o  u +  \Pi_k^o u  - \overline{u}_h + \overline{u}_h- u_h$, hence the convergence estimates for $\bm q - \bm q_h$ and $u - u_h$ in the main result follow from the triangle inequality, the estimate in \Cref{lemmainter}, the \Cref{error_dual2} and \Cref{lemma:global_Lip_error_estimate}. The superconvergent estimate for $ u - u_h^\star $ in the main result follows follow from the triangle inequality, \Cref{lemma1:HHO_proj}, \Cref{colorray_ustar}, and the estimates in \Cref{error_dual2} and \Cref{lemma:global_Lip_error_estimate}. }

	}
	
	\subsubsection{Error estimates for a local Lipschitz condition}
	\label{local_Lipschitz}
	In this section we assume that the nonlinearity $F$ is only locally Lipschitz, as is the case in many applications. To deal with this case, we assume that
	the mesh $\mathcal{T}_h$ is quasi-uniform.

	\begin{lemma}\label{eqn:loc_Lip_main_estimate}
		Assume the 	mesh $\mathcal{T}_h$ is quasi-uniform, and $d\in [2,2k+4)$ if $ (k,\ell) \neq (1,0) $ or $d\in [2,2k+2)$ if $ (k,\ell) = (1,0) $. Then for $h$ small enough and $t\in (0,T]$, the error estimates of \Cref{lemma:global_Lip_error_estimate} hold.
		%
		%\begin{align*}
		%\hspace{1em}&\hspace{-1em}
		%\|e_h^{u}\|^2_{\mathcal{T}_h}
		%+\|e_h^{\bm q}\|^2_{\mathcal{T}_h}
		%+\|h^{-1/2}_K( \Pi^{\partial}_k e^{u^\star}_h -e^{\widehat u}_h)\|_{\partial\mathcal{T}_h}^2+\int_0^t\left(\|\partial_te_h^{u}\|^2_{\mathcal{T}_h}+\|e_h^{\bm q}\|^2_{\mathcal{T}_h}+\|h_K^{-1/2}(\Pi^{\partial}_ke_h^{u^{\star}} - e_h^{\widehat u})\|_{\partial\mathcal T_h}^2\right) \nonumber\\
		%
		%
		%& \le  C\int_0^t\left(\|\partial_t(\Pi^o_{\ell}u-\overline{u}_h)\|^2_{\mathcal{T}_h}+
		%\|F(u)-\mathcal I_h F(u)\|_{\mathcal{T}_h}^2+\|u-\Pi_{k+1}^\star u\|_{\mathcal{T}_h}^2 + \|u-\mathcal I_h u\|_{\mathcal{T}_h}^2\right)\nonumber\\
		%&\quad +C\int_0^t\left(
		%\|\overline{u}_h-\Pi_{\ell}^o u\|_{\mathcal T_h}+h^2\|\overline{\bm{q}}_h-\bm{\Pi}^o_k\bm{q}\|_{\mathcal T_h} + 
		%h^2\|\bm{q}-\bm\Pi_k^o\bm{q}\|^2_{\math%cal T_h} \right).
	%	\end{align*}
	%
	\end{lemma}
	To prove this result, we are going to use the following auxiliary result. Its proof is in the Appendix.
	\begin{lemma} \label{infi} 
		We have
		\begin{align*}
		\| \Pi_{k+1}^{\star}u - u\|_{0,\infty,K}\le  C h_K\|\nabla u\|_{0,\infty,K}.%\label{pixin-infi}
		\end{align*}
	\end{lemma}
	
	\begin{proof}[Proof of \Cref{eqn:loc_Lip_main_estimate}]
		By \Cref{infi}, there is an $h_0$ such that for all $h\in (0,h_0]$ and for all $t\in[0,T]$, there holds
		\begin{align*}
		\|u-\Pi_{k+1}^{\star}u\|_{0,\infty,\mathcal{T}_h}&\le \frac{\delta}{2}.
		\end{align*}
		Therefore, $\Pi_{k+1}^{\star}u\in [-(M-\delta/2),(M-\delta/2)]$, and this implies that
		\begin{align*}
		\|F(u)-F(\Pi_{k+1}^{\star}u)\|_{\mathcal{T}_h}&\le L\|u-\Pi_{k+1}^{\star}u\|_{\mathcal{T}_h}.
		\end{align*}
		By an inverse inequality {and the assumption of quasiuniformity of the mesh,} we get
		% By \eqref{xin-0} and $d\in [2,2k+4)$, we have 
		\begin{align*}
		\|\Pi_{k+1}^\star u(0) -  {u}_h^\star(0)\|_{0,\infty,\mathcal T_h} \le h^{-d/2}\|\Pi_{k+1}^\star u(0) -  {u}_h^\star(0)\|_{\mathcal T_h}  \le C h^{-d/2}(h^{\ell+2 + \min\{\ell,1\}} + h^{k+2}),
		\end{align*}
		by  \Cref{super_con,error_0}.
		By the restrictions on $ d $, the upper bound of this error at time zero can be made strictly smaller than $ \delta/2 $ by taking $ h $ sufficiently small, say, for all $ h \in (0,h_0^\star] $, where $ h_0^\star \le h_0 $.
		
		Then, for each $h\in (0,h_0^\star]$ let $t_h\in (0,T]$ be the largest value such that for all $t\in[0,t_h]$ there holds
		\begin{align}
		\|\Pi_{k+1}^{\star}u- {u}_h^{\star}\|_{0,\infty,\mathcal{T}_h}\le \frac{\delta}{2}.\label{as2}
		\end{align}
		Therefore, $ {u}_h^{\star}\in [-M,M]$, and again we have
		\begin{align*}
		\|F(\Pi_{k+1}^{\star}u)-F( {u}_h^{\star})\|_{\mathcal{T}_h}\le L\|\Pi_{k+1}^{\star}u- {u}_h^{\star}\|_{\mathcal{T}_h}.
		\end{align*}
		Now the error estimate of   \Cref{eqn:loc_Lip_main_estimate} can be proved in exactly the same way as in \Cref{lemma:global_Lip_error_estimate}. However, the estimate now holds only for all $h\in (0,h_0^\star]$ and for all $t\in[0,t_h]$.

		By \Cref{super_con} and the error estimate, we have
		\begin{align*}
		\|\Pi_{k+1}^{\star}u(t_h)-{u}_h(t_h)\|_{\mathcal{T}_h}=\|e_h^{u}(t_h)\|_{\mathcal{T}_h}\le Ch^{\ell + 2+\min\{1,\ell\}} + Ch^{k+2}.
		\end{align*}
		By an inverse inequality we have
		\begin{align*}
		\|\Pi_{k+1}^{\star}u(t_h)-{u}_h(t_h)\|_{0,\infty,\mathcal{T}_h}
		\le C(h^{\ell + 2+\min\{1,\ell\}} + Ch^{k+2})h^{-\frac d 2}.
		\end{align*}
		As before, there exists $h_1 \in (0,h_0^\star]$ such that for all $h\in (0,h_1]$ there holds
		%	\begin{align*}
		%	C(h^{\ell + 2+\min\{1,\ell\}} + Ch^{k+2})h^{-\frac d 2} <\frac{\delta}{2}.
		%	\end{align*}
		%	Then for any $h\in(0,\min(h_1,h_2)]$, there holds
		\begin{align*}
		\|\Pi_{k+1}^{\star}u(t_h)-{u}_h(t_h)\|_{0,\infty,\mathcal{T}_h}<\frac{\delta}{2}.
		\end{align*}
		Since for each $ h \in (0,h_1] $ we have that $t_h \in (0,T]$ is the largest value such that \eqref{as2} holds for all $t\in[0,t_h]$, therefore $t_h=T$ for all $ h $ small enough.  This completes the proof.
		%then the desired results followed by  taking $t_h=T$ in \eqref{es2}.
	\end{proof}
	\subsection{Conclusion}
	\label{subsec:basic_projections}
	We can now conclude the proof of the main result. 
	To do that, we are going to need the following results. 
	
	The following error estimates for the $ L^2-$projections and the elementwise interpolation operator $ \mathcal I_h $ from \Cref{sec:HDG} are standard and can be found in \cite{MR2373954}.

	\begin{lemma}\label{lemmainter}
		Suppose $k, \ell \ge 0$. There exists a constant $C$ independent of $K\in\mathcal T_h$ such that
		%\begin{subequations}
			\begin{align*}
			&\|w - \mathcal  I_h w\|_K \le Ch^{k+2} |w|_{k+2,K}  &  &\forall \; w\in C(\bar K)\cap H^{k+2}(K), %\label{lemmainter_inter}
			\\
			&\|w - \Pi_{\ell}^o  w\|_K \le Ch^{\ell+1} |w|_{\ell+1,K}  &  &\forall \; w\in H^{\ell+1}(K), %\label{lemmainter_orthoo}
			\\
			&\|w- \Pi_k^\partial  w\|_{\partial K} \le Ch^{k+1/2} |w|_{k+1,K}  &  &\forall \; w\in H^{k+1}(K). %\label{lemmainter_orthoe}
			\end{align*}
		%\end{subequations}
	\end{lemma}
	
	%Take the partial derivative of \eqref{elliptic-projection} with respect to $t$, hence, $(\partial_t \overline{\bm q}_h,\partial_t\overline{u}_h,\partial_t\widehat{\overline u}_h)\in \bm V_h\times W_h\times M_h$ is the solution of
	%\begin{subequations}
	%	\begin{align}
		%
	%	(\partial_t\overline{\bm{q}}_h,\bm{r}_h)_{\mathcal{T}_h}-(\partial_t\overline{u}_h,\nabla\cdot \bm{r}_h)_{\mathcal{T}_h}+\langle\partial_t\widehat{\overline{u}}_h,\bm r_h\cdot\bm n \rangle_{\partial{\mathcal{T}_h}} &= 0, \label{pi3}\\
		%
	%	(\nabla\cdot\partial_t\overline{\bm{q}}_h, v_h)_{\mathcal{T}_h}
	%	-\langle \partial_t\overline{\bm q}_h\cdot\bm n,\widehat{v}_h \rangle_{\partial{\mathcal{T}_h}} \quad&
	%	\nonumber\\+\langle
	%	h^{-1}_K( \Pi^{\partial}_k\partial_t \overline{u}_h^\star -\partial_t\widehat{\overline{u}}_h),\Pi^{\partial}_kv_h^\star-\widehat{v}_h\rangle_{\partial{\mathcal{T}_h}} &=(f_t - \partial_{tt} u - F'(u)\partial_t u, v_h)_{\mathcal{T}_h},\label{pi4}
	%	\end{align}
	%\end{subequations}
	%for all $(\bm r_h,v_h,\widehat{v}_h)\in \bm V_h\times W_h\times M_h$.
We also need the following result. Its proof is given in \Cref{AppendixA}.
	\begin{theorem}\label{error_dual2}
		For any $t\in[0,T]$, we have the following error estimates
		%\begin{subequations}
			\begin{align*}
			\|\Pi_{k}^o \bm q - \overline{\bm q}_h\|_{\mathcal T_h}  &\le  C
			h\|\Pi_{\ell}^{o}(-\Delta  u) +\Delta  u\|_{\mathcal{T}_h}\nonumber\\
			&\quad+C(h^{1/2}\|\bm \Pi^o_k\bm{q}-\bm q\|_{\partial\mathcal{T}_h}+
			\|h_K^{-1/2}(\Pi_{k+1}^\star u -u)\|_{\partial{\mathcal{T}_h}}),\\
			\|\Pi_{\ell}^ou-\overline{u}_h\|_{\mathcal{T}_h}  &\le  C
			h^{1+\min\{1,\ell\}}\|\Pi_{\ell}^{o}(-\Delta  u)+\Delta  u\|_{\mathcal{T}_h}\nonumber\\
			&\quad+C(h^{1/2}\|\bm \Pi^o_k\bm{q}-\bm q\|_{\partial\mathcal{T}_h}+
			\|h_K^{-1/2}(\Pi_{k+1}^\star u -u)\|_{\partial{\mathcal{T}_h}}),\\
			\|\partial_t\Pi_{\ell}^ou-\partial_t\overline{u}_h\|_{\mathcal{T}_h} &\le Ch^{1+\min\{1,\ell\}}
			\|\Pi_{\ell}^{o}(-\Delta  u_t)+\Delta  u_t\|_{\mathcal{T}_h}\nonumber\\
			&\quad +C(h^{1/2}\|\bm \Pi^o_k\bm{q}_t-\bm q_t\|_{\partial\mathcal{T}_h}+
			\|h_K^{-1/2}(\Pi_{k+1}^\star u_t -u_t)\|_{\partial{\mathcal{T}_h}}).
			%\right).
			\end{align*}
		%\end{subequations}
	\end{theorem}
	
	We are now ready to conclude the proof of our main result. Indeed, if the nonlinearity is globally Lipschitz,
	since $ \bm q - \bm q_h = \bm q - \bm \Pi_k^o \bm q + \bm \Pi_k^o \bm q  - \overline{\bm q}_h + \overline{\bm q}_h- \bm q_h$ and $u-u_h = u -  \Pi_k^o  u +  \Pi_k^o u  - \overline{u}_h + \overline{u}_h- u_h$, the convergence estimates for $\bm q - \bm q_h$ and $u - u_h$ in the main result follow from the triangle inequality, the estimates in \Cref{lemmainter}, \Cref{error_dual2}, and \Cref{lemma:global_Lip_error_estimate}. The {superconvergence} estimate for $ u - u_h^\star $ in the main result follows from the triangle inequality, \Cref{lemma1:HHO_proj},  \Cref{super_con}, and the estimates in \Cref{error_dual2} and \Cref{lemma:global_Lip_error_estimate}.
	
	If the nonlinearity is locally Lipschitz, the estimates of the main result in this case now follow from the above result in the same way.
	This concludes the proof of the main result, \Cref{main_err_qu}.
	
	\section{Numerical Results}
	\label{sec:numerics}

	%In this section, we present  an example to confirm the theoretical results of  the Interpolatory HDG (ABC) methods.  The domain is the unit square $\Omega = [0,1]\times [0,1]\subset \mathbb R^2$ in 2D.

	%\begin{example}[The Chaffee-Infante equation]
	We test the Chaffee-Infante equation with an exact solution to illustrate the convergence theory.  The domain is the unit square $\Omega = (0,1)\times (0,1)\subset \mathbb R^2$, the nonlinear term is $F( u) := u^3-u$, and the source term $f$ is chosen so that the exact solution is $u = \sin(t)\sin(\pi x)\sin(\pi y)$. The meshes are uniform and made of triangles. The Crank-Nicolson method is used for the time discretization. {The initial condition is the simple {$L^2$}-projection of $u_0$ into $W_h$.} For Interpolatory HDG (AB), the time step is chosen as $\Delta t = h$ when $k=0$ and $\Delta t   = h^2$ when $k=1$, where $k$ is the polynomial degree. We choose $\Delta t=h$ when $k=1$ and $\Delta t = h^2$ when $k=2$ for Interpolatory HDG (C). We report the errors at the final time $ T = 1 $ in \Cref{table_1}.  The observed convergence rates match the theory.
	\begin{table}%[H]
		\caption{History of convergence.}\label{table_1}
		\centering
		%%%%%%%%%%%%%%%%%%%%%%%%%%%%%%%55
		Errors for $\bm{q}_h$, $u_h$ and $u_h^\star$ of HDG (A){
			\begin{tabular}{c|c|c|c|c|c|c|c}
				\Xhline{1pt}

				\multirow{2}{*}{Degree}
				&\multirow{2}{*}{$\frac{h}{\sqrt{2}}$}	
				&\multicolumn{2}{c|}{$\|\bm{q}-\bm{q}_h\|_{0,\Omega}$}	
				&\multicolumn{2}{c|}{$\|u-u_h\|_{0,\Omega}$}	
				&\multicolumn{2}{c}{$\|u-u_h^\star\|_{0,\Omega}$}	\\
				\cline{3-8}
				& &Error &Rate
				&Error &Rate
				&Error &Rate
				\\
				\cline{1-8}
				\multirow{5}{*}{ $k=0$}
				&$2^{-1}$	&1.18		&	    &2.93E-01	&	    &2.93E-01		&\\
				&$2^{-2}$	&6.33E-01	&0.89 	&9.53E-02	&1.62 	&9.53E-02		&1.62\\
				&$2^{-3}$	&3.23E-01	&0.97 	&2.47E-02	&1.95 	&2.47E-02	&1.95\\
				&$2^{-4}$	&1.62E-01	&0.99 	&6.24E-03	&1.99 	&6.24E-03	&1.99\\
				&$2^{-5}$	&  8.12E-02	&0.98 	&1.56E-03	&2.00 	&1.56E-03		&2.00\\

				\cline{1-8}
				\multirow{5}{*}{$k=1$}
				
				&$2^{-1}$	& 3.39E-02	&	    &  8.81E-02	&	    & 8.81E-02	&\\
				&$2^{-2}$	&9.15E-03	&1.97 	&1.14E-02	&2.95 	&1.14E-02	&2.95\\
				&$2^{-3}$	&2.33E-02	&1.99 	&1.44E-03	&3.00 	&1.44E-03		&3.00\\
				&$2^{-4}$	&5.86E-03	&1.99 	&1.80E-04	&3.00 	&1.80E-04	&3.00\\
				&$2^{-5}$	&1.47E-03	&2.00 	&2.25E-05	&3.00 	&2.25E-05		&3.00\\
				
				\Xhline{1pt}

			\end{tabular}
		}
		\bigskip
		
		%%%%%%%%%%%%%%%%%%%%%%%%%%%%%%%55
		Errors for $\bm{q}_h$, $u_h$ and $u_h^\star$ of HDG (B){
			\begin{tabular}{c|c|c|c|c|c|c|c}
				\Xhline{1pt}

				\multirow{2}{*}{Degree}
				&\multirow{2}{*}{$\frac{h}{\sqrt{2}}$}	
				&\multicolumn{2}{c|}{$\|\bm{q}-\bm{q}_h\|_{0,\Omega}$}	
				&\multicolumn{2}{c|}{$\|u-u_h\|_{0,\Omega}$}	
				&\multicolumn{2}{c}{$\|u-u_h^\star\|_{0,\Omega}$}	\\
				\cline{3-8}
				& &Error &Rate
				&Error &Rate
				&Error &Rate
				\\
				\cline{1-8}
				\multirow{5}{*}{ $k=0$}
				&$2^{-1}$	&1.21		&	    &3.23E-01	&	    &2.41E-01	&\\
				&$2^{-2}$	&6.40E-01	&0.92	&1.41E-01	&1.20 	&6.47E-01	&1.90\\
				&$2^{-3}$	& 3.24E-01	&0.98	&6.68E-01	&1.08 	&1.66E-02	&1.97\\
				&$2^{-4}$	&1.62E-01	&1.00	& 3.29E-02	&1.02 	& 4.17E-03	&2.00\\
				&$2^{-5}$	&8.13E-02	&1.00 	&1.64E-02	&1.00 	& 1.04E-03	&2.00\\

				\cline{1-8}
				\multirow{5}{*}{$k=1$}
				&$2^{-1}$	&3.41E-01	& 	    &9.33E-01	& 	    &6.31E-02	&\\
				&$2^{-2}$	&9.02E-02	&1.90	&2.12E-02	&2.14 	&9.05E-03	&2.80\\
				&$2^{-3}$	&2.28E-02	&1.98 	&5.07E-02	&2.07 	& 1.16E-03	&2.96\\
				&$2^{-4}$	&5.73E-03	&2.00 	& 1.25E-03	&2.02	&1.46E-04	&2.99\\
				&$2^{-5}$	&1.43E-03	&2.00 	&3.11E-04	&2.00 	&1.83E-05	&3.00\\
				
				\Xhline{1pt}

			\end{tabular}
		}
		
		\bigskip
		
		%%%%%%%%%%%%%%%%%%%%%%%%%%%%%%%55
		%%%%%%%%%%%%%%%%%%%%%%%%%%%%%%%55
		Errors for $\bm{q}_h$, $u_h$ and $u_h^\star$ of HDG (C){
			\begin{tabular}{c|c|c|c|c|c|c|c}
				\Xhline{1pt}

				\multirow{2}{*}{Degree}
				&\multirow{2}{*}{$\frac{h}{\sqrt{2}}$}	
				&\multicolumn{2}{c|}{$\|\bm{q}-\bm{q}_h\|_{0,\Omega}$}	
				&\multicolumn{2}{c|}{$\|u-u_h\|_{0,\Omega}$}	
				&\multicolumn{2}{c}{$\|u-u_h^\star\|_{0,\Omega}$}	\\
				\cline{3-8}
				& &Error &Rate
				&Error &Rate
				&Error &Rate
				\\
				\cline{1-8}
				\multirow{5}{*}{ $k=1$}
				&$2^{-1}$	&6.28E-01	&		&2.58E-01	&	    &1.16E-01		&\\
				&$2^{-2}$	&1.78E-01	&1.82 	&1.32E-01	&0.97	&3.20E-02		&1.86\\
				&$2^{-3}$	&4.58E-02	&1.96 	&6.56E-02	&1.00 	&8.24E-02		&1.96\\
				&$2^{-4}$	&1.15E-02	&1.99 	&3.28E-02	&1.00 	&2.07E-03		&1.99\\
				&$2^{-5}$	& 2.89E-03	&2.00 	&1.64E-02	&1.00 	&5.20E-04		&2.00\\

				\cline{1-8}
				\multirow{5}{*}{$k=2$}
				
				&$2^{-1}$	& 1.06E-01	&	    &7.39E-02	&	    & 1.27E-02	&\\
				&$2^{-2}$	&1.44E-02	&2.88 	&1.95E-02	&1.92 	&9.39E-04	&3.76\\
				&$2^{-3}$	&1.85E-03	&2.96 	&4.95E-03	&1.98 	&6.18E-05	&3.92\\
				&$2^{-4}$	&2.33E-04	&2.99 	&1.24E-03	&1.99 	&3.92E-06	&3.98\\
				&$2^{-5}$	&2.93E-05	&3.00 	&3.11E-04	&2.00 	&2.47E-07	&4.00\\

				\Xhline{1pt}

			\end{tabular}
		}
	\end{table}

	\section{Conclusion}
	
	In \cite{ChenCockburnSinglerZhang1},  we proposed a superconvergent Interpolatory HDG  method to  approximate  the solution of nonlinear reaction diffusion PDEs. The new  method uses a postprocessing procedure along with an  interpolation operator  to evaluate the nonlinear term.  This simple change recovers  the superconvergence  that was  lost  in our earlier Interpolatory HDG work \cite{CockburnSinglerZhang1}. Furthermore,  this method retains  the computational advantages of our Interpolatory HDG method from  \cite{CockburnSinglerZhang1}.
	
	We extended the idea developed previously and devised superconvergent Interpolatory HDG methods inspired by hybrid high-order methods \cite{MR3507267}. We proved that the interpolatory procedure does not reduce the convergence rate. %Moreover, some of the  methods  superconverge for all polynomial degrees $k\ge 0$, while we only obtained  superconvergence rate for $k\ge 1$ in Part I. 
	
	The devising of superconvergent HDG methods for equations with the more general nonlinear 
	term $F(\nabla u, u)$ constitutes a subject of ongoing work.

	\appendix

	\section{Approximation estimates of auxiliary projections}
		\label{HHOProjections}

	\subsection{Proof of  \Cref{infi}}
	Here we prove the estimate for $\Pi_{k+1}^\star u- u$ in \Cref{infi}.
	
	We are going to use the following auxiliary result.
	
	\begin{lemma}\label{lemma2:HHO_proj} For any $K\in\mathcal{T}_h$, we have
		\begin{eqnarray*}
		\|\Pi_{k+1}^\star u-u\|_{0,K}
		\le C\left( h_K\|\nabla u-\nabla\Pi^{o}_{k+1}u\|_{K}+\| u-\Pi^{o}_{k+1}u\|_{K}\right).
		%\label{es_Pixin}
		\end{eqnarray*}
	\end{lemma}
	\begin{proof} By definitions \eqref{pixin} and \eqref{post}, we obtain
		%\begin{subequations}
			\begin{align*}
			(\nabla\Pi_{k+1}^\star u,\nabla z_h)_K&=-(\Pi^{o}_\ell u,\Delta z_h)_K+\langle\Pi^{\partial}_k u,\bm n\cdot\nabla z_h \rangle_{\partial K},\\
			(\Pi_{k+1}^\star u,w_h)_K&=(\Pi^{o}_\ell u,w_h)_K,
			\end{align*}
		%\end{subequations}
		for all $(z_h,w_h)\in [\mathcal {P_{\ell}^{k+1}(K)]^\perp} \times \mathcal{P}^{\ell}(K) $. This leads to
		%\begin{subequations}
			\begin{align*}
			(\nabla\Pi_{k+1}^\star u,\nabla z_h)_K&=(\nabla u,\nabla z_h)_K,%\label{or1}
			\\
			(\Pi_{k+1}^\star u,w_h)_K&=(\Pi^{o}_{k+1} u,w_h)_K.%\label{or2}
			\end{align*}
		%\end{subequations}
		The last equation implies that  {$\Pi_{k+1}^\star u-\Pi^{o}_{k+1}u\in  [\mathcal P_{\ell}^{k+1}(K)]^{\perp}$} and so, we can then take
		$z_h:=\Pi_{k+1}^\star u-\Pi^{o}_{k+1}u$ in the first equation to get		\begin{align*}
		\|\nabla\Pi_{k+1}^\star u-\nabla\Pi^{o}_{k+1}u\|^2_{K}=(\nabla\Pi_{k+1}^\star u-\nabla\Pi^{o}_{k+1}u, \nabla u-\nabla\Pi^{o}_{k+1}u)_K,
		\end{align*}
		and
		\begin{align*}
		\|\nabla\Pi_{k+1}^\star u-\nabla\Pi^{o}_{k+1}u\|_{K}\le\|\nabla u-\nabla\Pi^{o}_{k+1}u\|_{K}.
		\end{align*}
		Since {$\Pi_{k+1}^\star u-\Pi^{o}_{k+1}u\in  [\mathcal P_{\ell}^{k+1}(K)]^{\perp}$}, we have
		\begin{align*}
		(\Pi_{k+1}^\star u-\Pi^{o}_{k+1}u,1)_K=0,
		\end{align*}
		and using Poincar\'e's inequality, we obtain
		\begin{align*}%\label{last}
		\|\Pi_{k+1}^\star u-\Pi^{o}_{k+1}u\|_{K} \le C h_K\|\nabla\Pi_{k+1}^\star u-\nabla\Pi^{o}_{k+1}u\|_{K} \le C h_K\|\nabla u-\nabla\Pi^{o}_{k+1}u\|_{K}.
		\end{align*}
		Then the estimate follows by applying the triangle inequality.	{This completes the proof.  } \end{proof}
	
	We are now ready to prove  \Cref{infi}.
	Using inverse inequalities, Poincar\'e's inequality, and the approximation properties for $\Pi_{k+1}^o$, one gets
		\begin{align*}
		\| u-\Pi_{k+1}^{\star}u\|_{0,\infty,K}&\le
		\| \Pi_{k+1}^{\star}u-\Pi_{k+1}^o u\|_{0,\infty,K}
		+ C\|\Pi_{k+1}^o u-u\|_{0,\infty,K}
		\nonumber\\
		&\le
		Ch_K^{-d/2}\| \Pi_{k+1}^{\star}u-\Pi_{k+1}^o u\|_{0,K}
		+C h_K\|\nabla u\|_{0,\infty,K}
		\nonumber\\
		%	&\le
		%	Ch_K^{1-d/2}\| \nabla(\Pi_{k+1}^{\star}u-\Pi_{k+1}^o u)\|_{0,K}
		%	+ Ch_K\|\nabla u\|_{0,\infty,K}
		%	\nonumber\\
		&\le
		Ch_K^{1-d/2}| u-\Pi_{k+1}^o u|_{1,K}
		+C h_K\|\nabla u\|_{0,\infty,K}
		\nonumber\\
		&\le Ch_K^{1-d/2} h_K^{d/2-1} |\widehat{u}-\widehat{\Pi}_{k+1}^{o}\widehat{u}|_{1,\widehat{K}} + C h_K\|\nabla u\|_{0,\infty,K}\\
		&
		\le C|\widehat{u}|_{1,\widehat{K}} + C h_K\|\nabla u\|_{0,\infty,K}\\
		&
		\le C|\widehat{u}|_{1,\infty,\widehat{K}} + C h_K\|\nabla u\|_{0,\infty,K}\\
		&
		\le Ch_K|{u}|_{1,\infty,{K}} + C h_K\|\nabla u\|_{0,\infty,K}.
		\end{align*}
		Here, we used a standard scaling argument and $\widehat K$ is the reference element.
	{This completes the proof of  \Cref{infi}.  }

\subsection{Proof of  \Cref{super_con}}	
Here, we prove the estimate for $\Pi_{k+1}^\star u - u_h^\star$ in \Cref{super_con}.

	%Using $u_h - \Pi_{\ell}^o u =u_h  - \overline u_h + \overline u_h - \Pi_{\ell}^o u $, $\bm q_h - \Pi_{k}^o \bm q =\bm q_h  - \overline {\bm q}_h + \overline {\bm q}_h  - \Pi_{k}^o \bm q $, Lemma \ref{colorray_ustar} follows immediately from this result.

	%Let us prove Lemma \ref{super_con}.

		Let {$z_h\in [\mathcal{P}_{\ell}^{k+1}(K)]^{\perp}$} and take $\bm r_h=\nabla z_h$ in the first equation of  \Cref{lemma1:HHO_proj} to get
		\begin{align*}
		(\bm{q}_h,\nabla z_h)-(u_h,\Delta  z_h)_{\mathcal{T}_h}+\left\langle\widehat{u}_h,\nabla z_h\cdot\bm n \right\rangle_{\partial{\mathcal{T}_h}} &= 0.%\label{ho}
		\end{align*}
		Combined with \eqref{pp1} one gets
		\begin{align*}
		(\nabla u_h^{\star},\nabla z_h)=-(\bm{q}_h,\nabla z_h) \quad \forall \; z_h\in [\mathcal{P}_{\ell}^{k+1}(K)]^{\perp}.%\label{equal}
		\end{align*}
		By the definition of $\Pi_{k+1}^{\star}$, as in the proof of \Cref{lemma1:HHO_proj} one gets
		\begin{align*}
		( \nabla \Pi_{k+1}^{\star}u ,\nabla z_{h} )_K&=-(\Pi^{o}_\ell u,\Delta z_h)_K+\langle\Pi^{\partial}_k u,\bm n\cdot\nabla z_h \rangle_{\partial K}=(\nabla u,\nabla z_h)_K.%\label{equal_2}
		\end{align*}
		Let $e_h=u_h^\star - u_h+\Pi_{\ell}^o u-\Pi_{k+1}^{\star} u$, and then {$e_h\in [\mathcal{P}_{\ell}^{k+1}(K)]^{\perp}$}.  By the two previous equations, $ \bm q = -\nabla u $, and an inverse inequality we have
		\begin{align*}
		\|\nabla e_{h}\|_K^2&=(\nabla (u_h^\star - u_h),\nabla e_{h} )_K+( \nabla (\Pi_{\ell}^o u-\Pi_{k+1}^{\star} u),\nabla e_{h} )_K \nonumber\\
		&=(-\bm{q}_h-\nabla u_h,\nabla e_{h} )_K+(  \nabla (\Pi_{\ell}^o u-u),\nabla e_{h} )_K \nonumber\\
		&=((\bm{q}-\bm{\Pi}^o_k\bm{q})-(\bm{q}_h-\bm{\Pi}^o_k\bm{q})  +\nabla (\Pi_{\ell}^o u-u_h),\nabla e_{h})_K \nonumber\\
		&\le C (h_K^{-1}\|u_h-\Pi_{\ell}^o u\|_K+\|\bm{q}_h-\bm{\Pi}^o_k\bm{q}\|_K + \|\bm{q}-\bm{\Pi}^o_k\bm{q}\|_K )\|\nabla e_{h}\|_K.%\label{H1}
		\end{align*}
		Since $(e_h,1)_K=0$,  we can now apply the Poincar\'{e} inequality to get 
		\begin{align*}
		\|e_{h}\|_K \le C h_K \|\nabla e_h\|_K \le C (\|u_h-\Pi_{\ell}^o u\|_K+h_K\|\bm{q}_h-\bm{\Pi}^o_k\bm{q}\|_K + h_K\|\bm{q}-\bm{\Pi}^o_k\bm{q}\|_K).
		\end{align*}
		This means
		\begin{align*}
		\|e_{h} \|_{\mathcal T_h} \le C( \|u_h-\Pi_{\ell}^o u\|_{\mathcal T_h}+h\|\bm{q}_h-\bm{\Pi}^o_k\bm{q}\|_{\mathcal T_h} +h\|\bm{q}-\bm{\Pi}^o_k\bm{q}\|_{\mathcal T_h} ).
		\end{align*}
		Hence, we have
		\begin{align*}
		\|\Pi_{k+1}^{\star} u - u_h^\star\|_{\mathcal T_h}&\le\|\Pi_{k+1}^{\star} u -\Pi^{o}_{\ell}  u- u_h^\star + u_h\|_{\mathcal T_h} + \|\Pi^{o}_{\ell} u-u_h\|_{\mathcal T_h} \nonumber\\
		&  \le C ( \|u_h-\Pi_{\ell}^o u\|_{\mathcal T_h}+h\|\bm{q}_h-\bm{\Pi}^o_k\bm{q}\|_{\mathcal T_h} +h\|\bm{q}-\bm{\Pi}^o_k\bm{q}\|_{\mathcal T_h} ).
		\end{align*}
	{This completes the proof of  \Cref{super_con}.  }
	
	\section{Proof of  \Cref{error_dual2}}\label{AppendixA}
	This appendix is devoted to the proof of
	the approximation estimates of  
	\Cref{error_dual2}.
	We only give the proofs of the estimates for $\|\bm\Pi_{k}^o \bm q - \overline{\bm q}_h\|_{\mathcal T_h}$ { and  $\|\Pi_{\ell}^ou-\overline{u}_h\|_{\mathcal{T}_h}$.}
	%, and  $\| \Pi_{k+1}^\star u - {\overline u_h}^\star\|_{\mathcal{T}_h}$.
	The proof  of the estimate for $\|\partial_t\Pi_{\ell}^ou-\partial_t\overline{u}_h\|_{\mathcal{T}_h}$ is very similar and is omitted. We use the notation
	\[
	\varepsilon_h^{\bm q}=\bm{\Pi}_{k}^o\bm q-\overline{\bm q}_h , 
	\quad\varepsilon_h^{ u}=\Pi_{\ell}^o u-\overline{u}_h,
	\quad
	\varepsilon_h^{\widehat{u}}=\Pi_{k}^{\partial}u-\widehat{\overline{u}}_h,
	\quad
	\text{ and }
	\quad
	\varepsilon_h^{ u^\star}=\Pi_{k+1}^\star u-\overline{u}_h^\star,
	\]
	and split the proof  into four steps.
	
	\subsection*{Step 1: Equations for the projections of the errors}
	\begin{lemma}\label{error_u}
		For all $(\bm r_h,v_h,\widehat v_h)\in \bm V_h\times W_h\times M_h$, we have
%		\begin{subequations}\label{error3}
			\begin{align*}
			(\varepsilon_h^{\bm q},\bm{r}_h)_{\mathcal{T}_h}-(\varepsilon_h^{u},\nabla\cdot \bm{r}_h)_{\mathcal{T}_h}+\langle\varepsilon_h^{\widehat{u}},\bm r_h\cdot\bm n\rangle_{\partial{\mathcal{T}_h}} &= 0,%\label{error2_a} 
			\\
			(\nabla\cdot\varepsilon_h^{\bm q}, v_h)_{\mathcal{T}_h}
			-\langle \varepsilon_h^{\bm q}\cdot\bm n,\widehat{v}_h \rangle_{\partial{\mathcal{T}_h}}
			+\langle
			h_K^{-1}( \Pi^{\partial}_k \varepsilon_h^{u^\star} -\varepsilon_h^{\widehat{u}}),\Pi^{\partial}_kv_h^\star-\widehat{v}_h\rangle_{\partial{\mathcal{T}_h}} &= RHS_h,
			\end{align*}
			where
			\begin{alignat*}{1}
			RHS_h:=&\;((\mathbb I-\Pi_{\ell}^o)(-\Delta u),(\mathbb {I}-\Pi_{\ell}^o)v_h^\star)
			+E_h(\bm q,u;v_h,\widehat{v}_h),%\label{error2_b}
			\\
			E_h(\bm q,u;v_h,\widehat{v}_h) 
		:=&-\langle( \bm \Pi^o_k\bm{q}-\bm q)\cdot\bm n,\widehat{v}_h-v_h^\star \rangle_{\partial{\mathcal{T}_h}} +\langle
		h_K^{-1}( \Pi_{k+1}^\star u -u),\Pi^{\partial}_k v_h^\star -\widehat{v}_h\rangle_{\partial{\mathcal{T}_h}},
			\end{alignat*}
	%	\end{subequations}
		 and $\mathbb I$ is the identity operator.
	\end{lemma}
	
	\begin{proof}
		We begin by noting that, by the properties of $\bm{\Pi}_k^o$, $\Pi_{\ell}^o$, and $\Pi_k^\partial$, we have
		\begin{align*}
		(\bm \Pi^o_k\bm{q},\bm{r}_h)_{\mathcal{T}_h}-(\Pi^o_{\ell} u,\nabla\cdot \bm{r}_h)_{\mathcal{T}_h}+\langle\Pi^{\partial}_{k}u,\bm r_h\cdot\bm n \rangle_{\partial{\mathcal{T}_h}}
		=(\bm{q},\bm{r}_h)_{\mathcal{T}_h}-( u,\nabla\cdot \bm{r}_h)_{\mathcal{T}_h}+\left\langle u,\bm r_h\cdot\bm n \right\rangle_{\partial{\mathcal{T}_h}}=0,
		\end{align*}
		since $\bm q+\nabla u=0$. Also, since $\langle  \bm q\cdot\bm n,\widehat{v}_h  \rangle_{\partial{\mathcal{T}_h}}=0$, we have
		\begin{align*}
		(\nabla\cdot\bm \Pi^o_k\bm{q}, v_h)_{\mathcal{T}_h}
		-\langle \bm \Pi^o_k\bm{q}\cdot\bm n,\widehat{v}_h \rangle_{\partial{\mathcal{T}_h}}&=((\nabla\cdot\bm \Pi^o_k\bm{q}, v_h^\star)_{\mathcal{T}_h}
		-\langle \bm \Pi^o_k\bm{q}\cdot\bm n,\widehat{v}_h \rangle_{\partial{\mathcal{T}_h}}\\
		&=		(\nabla\cdot\bm{q}, v^\star_h)_{\mathcal{T}_h}
		-\langle ( \bm\Pi^o_k\bm{q}-\bm q)\cdot\bm n,\widehat{v}_h -v_h^\star \rangle_{\partial{\mathcal{T}_h}}\\
		&=(-\Delta u, v^\star_h)_{\mathcal{T}_h}
		-\langle ( \bm\Pi^o_k\bm{q}-\bm q)\cdot\bm n,\widehat{v}_h -v_h^\star \rangle_{\partial{\mathcal{T}_h}}.
		\end{align*}
		As a consequence,
%		\begin{subequations}\label{projection_error2}
			\begin{align*}
			(\bm \Pi^o_k\bm{q},\bm{r}_h)_{\mathcal{T}_h}-(\Pi^o_{\ell} u,\nabla\cdot \bm{r}_h)_{\mathcal{T}_h}+\langle\Pi^{\partial}_{k}u,\bm r_h\cdot\bm n \rangle_{\partial{\mathcal{T}_h}} &= 0,%\label{projection_error2_a}\\
			\\
			(\nabla\cdot\bm \Pi^o_k\bm{q}, v_h)_{\mathcal{T}_h}-\langle \bm \Pi^o_k\bm{q}\cdot\bm n,\widehat{v}_h \rangle_{\partial{\mathcal{T}_h}} \quad %& \nonumber\\
			+\langle
			h_K^{-1}( \Pi^{\partial}_k \Pi_{k+1}^\star u -\Pi^{\partial}_{k}u),\Pi^{\partial}_kv_h^\star-\widehat{v}_h\rangle_{\partial{\mathcal{T}_h}} &= 
			(-\Delta u,v_h^\star)_{\mathcal{T}_h}\nonumber\\
			& \quad +E_h(\bm q,u;v_h,\widehat{v}_h).%\label{projection_error2_b}
			\end{align*}
%		\end{subequations}\
		The wanted equations can be now obtained by subtracting these equations from the equations defining the HDG elliptic approximation \eqref{elliptic-projection}.
	{This completes the proof.  } \end{proof}
	
	\subsection*{Step 2: Estimate for $\varepsilon_h^q$ by an energy argument}
	\label{sec:energy_argument_q}

	\begin{lemma}\label{theorem_err_u2}
		We have
		\begin{align*}
		\hspace{1em}&\hspace{-1em} \|\nabla\varepsilon_h^{u^\star}\|_{\mathcal T_h} +\|\varepsilon_h^{\bm q}\|_{\mathcal T_h} +\|h_K^{-1/2}(\Pi^{\partial}_k\varepsilon_h^{u^\star}-\varepsilon_h^{\widehat{u}})\|_{\partial\mathcal{T}_h}\\
		&\le C\left(
		h\|(\Pi_{\ell}^{o}-\mathbb I)(-\Delta u)\|_{\mathcal{T}_h}
		+h^{1/2}\|\bm \Pi^o_k\bm{q}-\bm q\|_{\partial\mathcal{T}_h}+
		\|h_K^{-1/2}(\Pi_{k+1}^\star u -u)\|_{\partial{\mathcal{T}_h}}
		\right).
		\end{align*}
	\end{lemma}
	{This result implies the estimate for the 
	approximate flux in \Cref{error_dual2}.}
	To prove this lemma, we need the following auxiliary result.
		\begin{lemma} \label{ener_q_ustar}
		We have
		\begin{subequations}
			\begin{align}
			\|\varepsilon_h^{\bm q}\|_{\mathcal{T}_h}& \le C\left( \|\nabla \varepsilon_h^{u^\star}\|_{\mathcal{T}_h}
			+\|h_K^{-1/2}(\Pi_{k}^{\partial}\varepsilon_h^{u^\star}-\varepsilon_h^{\widehat{u}})\|_{\partial\mathcal{T}_h}\right),\\
			\|\nabla \varepsilon_h^{u^\star}\|_{\mathcal{T}_h}& \le \left(  \|\varepsilon_h^{\bm q}\|_{\mathcal{T}_h}
			+\|h_K^{-1/2}(\Pi_{k}^{\partial}\varepsilon_h^{u^\star}-\varepsilon_h^{\widehat{u}})\|_{\partial\mathcal{T}_h}\right).
			%\label{eq2}
			\end{align}
		\end{subequations}
	\end{lemma}
	\begin{proof}
Using the first equation of  \Cref{error_u}, the definition of $\mathfrak{p}_h^{k+1}$ in \eqref{post}, and $\nabla\cdot\bm r_h\in W_h$,  we have
		\begin{align*}
		(\varepsilon_h^{\bm q},\bm{r}_h)_{\mathcal{T}_h}-(\varepsilon_h^{u^\star},\nabla\cdot \bm{r}_h)_{\mathcal{T}_h}+\langle\varepsilon_h^{\widehat{u}},\bm r_h\cdot\bm n \rangle_{\partial{\mathcal{T}_h}} = 0.
		\end{align*}
		Integration by parts gives
		\begin{align*}
		(\varepsilon_h^{\bm q},\bm{r}_h)_{\mathcal{T}_h}+(\nabla\varepsilon_h^{u^\star}, \bm{r}_h)_{\mathcal{T}_h}+\langle\varepsilon_h^{\widehat{u}}-\Pi_{k}^{\partial}\varepsilon_h^{u^\star},\bm r_h\cdot\bm n \rangle_{\partial{\mathcal{T}_h}} = 0.
		\end{align*}
		Since $\nabla \varepsilon_h^{u^*}\in \bm V_h$, by taking first $\bm r_h:=\varepsilon_h^{\bm q}$ and then $\bm r_h:=\nabla \varepsilon_h^{u^*}$, one gets
		\begin{align*}
		\|\varepsilon_h^{\bm q}\|_{\mathcal{T}_h}& \le C\left( \|\nabla \varepsilon_h^{u^\star}\|_{\mathcal{T}_h}
		+\|h_K^{-1/2}(\Pi_{k}^{\partial}\varepsilon_h^{u^\star}-\varepsilon_h^{\widehat{u}})\|_{\partial\mathcal{T}_h}\right),\\
		\|\nabla \varepsilon_h^{u^\star}\|_{\mathcal{T}_h}&\le C\left( \|\varepsilon_h^{\bm q}\|_{\mathcal{T}_h}
		+\|h_K^{-1/2}(\Pi_{k}^{\partial}\varepsilon_h^{u^\star}-\varepsilon_h^{\widehat{u}})\|_{\partial\mathcal{T}_h}\right),
		\end{align*}
		respectively. {This completes the proof.  } \end{proof}

	We can now prove  \Cref{theorem_err_u2}.
	\begin{proof}
	We take $(\bm r_h,v_h,\widehat{v}_h):=(\varepsilon_h^{\bm q},\varepsilon_h^{u},\varepsilon_h^{\widehat u})$ in the error equations of  \Cref{error_u}, and add them to get
	\[
	\|\varepsilon_h^{\bm q}\|^2_{\mathcal{T}_h}+\|h_K^{-1/2}( \Pi^{\partial}_k \varepsilon_h^{u^\star} -\varepsilon_h^{\widehat{u}})\|^2_{\partial{\mathcal{T}_h}}
	=  R_1+R_2+R_3,
	\]
	where
	\begin{alignat*}{1}
	R_1&:=
	((\mathbb I-\Pi_{\ell}^o)(-\Delta u), (\mathbb I-\Pi_{\ell}^o)\varepsilon_h^{u^\star})_{\mathcal{T}_{h}},
		\\
	R_2&:=-\langle( \bm \Pi^o_k\bm{q}-\bm q)\cdot\bm n,\varepsilon^{\widehat{u}}_h-\varepsilon_h^{u^\star} \rangle_{\partial{\mathcal{T}_h}}\nonumber
	\\
	R_3&:=\langle
	h_K^{-1}( \Pi_{k+1}^\star u -u),\Pi^{\partial}_k\varepsilon_h^{u^\star} -\varepsilon^{\widehat{u}}_h\rangle_{\partial\mathcal{T}_h}.
	\end{alignat*}
		Since
		\begin{align*}
		|R_1|&\le Ch\| (\mathbb I-\Pi_{\ell}^{o})(-\Delta u)\|_{\mathcal{T}_h}\|\nabla\varepsilon_h^{u^\star} \|_{\mathcal{T}_h}, \nonumber\\
		%
		%&\le Ch\| (\mathbb I-\Pi_{\ell}^{o})(-\Delta u)\|_{\mathcal{T}_h}
		%\left(  \|\varepsilon_h^{\bm q}\|_{\mathcal{T}_h}+\|h_K^{-1/2}( \Pi^{\partial}_k \varepsilon_h^{u^\star} -\varepsilon_h^{\widehat{u}})\|_{\partial{\mathcal{T}_h}} \right), \\
		%
		|R_2|&\le Ch^{1/2}\|\bm \Pi^o_k\bm{q}-\bm q\|_{\partial\mathcal{T}_h}
		\left(  \|\nabla\varepsilon_h^{u^\star} \|_{\mathcal{T}_h}+\|h_K^{-1/2}( \Pi^{\partial}_k \varepsilon_h^{u^\star} -\varepsilon_h^{\widehat{u}})\|_{\partial{\mathcal{T}_h}} \right), \nonumber\\
		%
		%&\le Ch^{1/2}\|\bm \Pi^o_k\bm{q}-\bm q\|_{\partial\mathcal{T}_h}
		%\left(  \|\varepsilon_h^{\bm q} \|_{\mathcal{T}_h}+\|h_K^{-1/2}( \Pi^{\partial}_k \varepsilon_h^{u^\star} -\varepsilon_h^{\widehat{u}})\|_{\partial{\mathcal{T}_h}} \right), \\
		%
		|R_3|&\le \|h_K^{-1/2}(\Pi_{k+1}^\star u -u)\|_{\partial{\mathcal{T}_h}}\|h_K^{-1/2}( \Pi^{\partial}_k \varepsilon_h^{u^\star} -\varepsilon_h^{\widehat{u}})\|_{\partial{\mathcal{T}_h}},
		\end{align*}
		using the last two estimates of \Cref{ener_q_ustar} and simple algebraic manipulations, we get the  desired result. 
	\end{proof}

	\subsection*{Step 3: Estimate for $\varepsilon_h^{u^\star}$ by a {duality} argument}
	
	\begin{lemma} Assume that the elliptic regularity inequality \eqref{regular} holds. Then, we have
		\begin{align*}
		{\|\varepsilon_h^{u^\star}\|_{\mathcal{T}_h}}
		%+\|\varepsilon_h^{u}\|_{\mathcal{T}_h} 
		&\le C h^{1+\min\{\ell,1\}}\|(\mathbb I- \Pi_{\ell}^o)(-\Delta u)\|_{\mathcal{T}_h}\nonumber\\
		&\quad+C(h^{3/2}\|\bm \Pi^o_k\bm{q}-\bm q\|_{\partial\mathcal{T}_h}+
		h\|h_K^{-1/2}(\Pi_{k+1}^\star u -u)\|_{\partial{\mathcal{T}_h}}).
		\end{align*}
	\end{lemma}

	\begin{proof} 
		{Setting $g:=\varepsilon_h^{u^\star}$ in the dual problem, and proceeding as in the proof of} \Cref{error_u}, we get 
		\begin{subequations}
			\begin{align}
			(\bm \Pi^o_k\bm{\Phi},\bm{r}_h)_{\mathcal{T}_h}-(\Pi^o_{\ell} \Psi,\nabla\cdot \bm{r}_h)_{\mathcal{T}_h}+\langle\Pi^{\partial}_{k}\Psi,\bm r_h\cdot\bm n \rangle_{\partial{\mathcal{T}_h}} &= 0, \label{dual_1}\\
			(\nabla\cdot\bm \Pi^o_k\bm{\Phi}, v_h)_{\mathcal{T}_h}
			-\langle \bm \Pi^o_k\bm{\Phi}\cdot\bm n,\widehat{v}_h \rangle_{\partial{\mathcal{T}_h}} \quad&
			\nonumber\\+\langle
			h_K^{-1}( \Pi^{\partial}_k \Pi_{k+1}^\star\Psi -\Pi^{\partial}_{k}\Psi),\Pi^{\partial}_kv_h^\star-\widehat{v}_h\rangle_{\partial{\mathcal{T}_h}} &= (\varepsilon_h^{u^\star},v^\star_h)_{\mathcal{T}_h}+E_h(\bm \Phi,\Psi;v_h,\widehat{v}_h),\label{dual_2}
			\end{align}
		\end{subequations}
		where
		\begin{align*}
		E_h(\bm \Phi,\Psi;v_h,\widehat{v}_h)=-\langle( \bm \Pi^o_k\bm{\Phi}-\bm \Phi)\cdot\bm n,\widehat{v}_h-v_h^\star \rangle_{\partial{\mathcal{T}_h}}+\langle
		h_K^{-1}( \Pi_{k+1}^\star{\Psi} -\Psi),\Pi^{\partial}_kv_h^\star-\widehat{v}_h\rangle_{\partial{\mathcal{T}_h}}.
		\end{align*}
		Then taking $(v_h,\widehat{v}_h):=(\varepsilon_h^u,\varepsilon_h^{\widehat{u}})$ in \eqref{dual_2}, we get
		\begin{align*}
		\|\varepsilon_h^{u^\star}\|^2_{\mathcal{T}_h}
		&=(\nabla\cdot\bm \Pi^o_k\bm{\Phi}, \varepsilon_h^u)_{\mathcal{T}_h}
		-\langle \bm \Pi^o_k\bm{\Phi}\cdot\bm n,\varepsilon_h^{\widehat{u}} \rangle_{\partial{\mathcal{T}_h}}
		\\
		&\quad+\langle
		h_K^{-1}( \Pi^{\partial}_k \Pi_{k+1}^\star\Psi -\Pi^{\partial}_{k}\Psi),\Pi^{\partial}_kv_h^\star-\widehat{v}_h\rangle_{\partial{\mathcal{T}_h}}{-E_h(\bm \Phi,\Psi;\varepsilon^{u}_h,\varepsilon^{\widehat{u}}_h)}
		\\
		&= (\varepsilon_h^{\bm q},\bm \Pi^o_k\bm{\Phi})_{\mathcal{T}_h} + \langle
		h_K^{-1}( \Pi^{\partial}_k \Pi_{k+1}^\star\Psi -\Pi^{\partial}_{k}\Psi),\Pi^{\partial}_k\varepsilon_h^{u^\star}-\varepsilon_h^{\widehat{u}}\rangle_{\partial{\mathcal{T}_h}}{-E_h(\bm \Phi,\Psi;\varepsilon^{u}_h,\varepsilon^{\widehat{u}}_h)},
		\intertext{{by the first equation of Lemma \ref{error_u} with $ \bm{r}_h:=\bm \Pi^o_k\bm{\Phi}$.}
		By \eqref{dual_1} {with $\bm{r}_h:=\varepsilon_h^{\bm q}$}, we obtain
		}
		\|\varepsilon_h^{u^\star}\|^2_{\mathcal{T}_h}&=
		(\Pi^o_{\ell} \Psi,\nabla\cdot \varepsilon_h^{\bm q})_{\mathcal{T}_h}\
		-\langle\Pi^{\partial}_{k}\Psi,\varepsilon_h^{\bm q}\cdot\bm n \rangle_{\partial{\mathcal{T}_h}}
		+\langle
		h_K^{-1}( \Pi^{\partial}_k \Pi_{k+1}^\star\Psi -\Pi^{\partial}_{k}\Psi),\Pi^{\partial}_k\varepsilon_h^{u^\star}-\varepsilon_h^{\widehat{u}}\rangle_{\partial{\mathcal{T}_h}}
		\\
		&\quad-E_h(\bm \Phi,\Psi;\varepsilon^{u}_h,\varepsilon^{\widehat{u}}_h)\nonumber\\
		&=({(\mathbb I-\Pi_{\ell}^o)(-\Delta u)},\Pi_{k+1}^\star\Psi-\Pi_{\ell}^o\Psi)+ E_h(\bm q,u;\Pi^{o}_{\ell}\Psi,\Pi^{\partial}_{k}\Psi)-E_h(\bm \Phi,\Psi;\varepsilon^{u}_h,\varepsilon^{\widehat{u}}_h),
		\end{align*}
		{by the second equation of Lemma \ref{error_u} with $(v_h,\widehat{v}_h):=(\Pi^{o}_{\ell}\Psi,\Pi^{\partial}_{k}\Psi)$.
		Inserting the definitions of the $E_h$-terms}, we finally get 
		\begin{alignat*}{2}
		\|\varepsilon_h^{u^\star}\|^2_{\mathcal{T}_h}
		&= ((\mathbb I-\Pi_{\ell}^o)(-\Delta u),\Pi_{k+1}^\star\Psi-\Pi_{\ell}^o\Psi)
		\\
		&\quad -\langle( \bm \Pi^o_k\bm{q}-\bm q)\cdot\bm n,\Pi^{\partial}_{k}\Psi-\Pi^\star_{k+1}\Psi \rangle_{\partial{\mathcal{T}_h}}
		&&+\langle
		h_K^{-1}( \Pi_{k+1}^\star u -u),\Pi^{\partial}_k\Pi^\star_{k+1}\Psi-\Pi^{\partial}_{k}\Psi\rangle_{\partial{\mathcal{T}_h}}
		\\
		&\quad{+\langle( \bm \Pi^o_k\bm{\Phi}-\bm \Phi)\cdot\bm n,\varepsilon_h^{\widehat{u}}-\varepsilon_h^{u^\star}\rangle_{\partial{\mathcal{T}_h}}}\nonumber
		&&-\langle
		h_K^{-1}( \Pi_{k+1}^\star{\Psi} -\Psi),\Pi^{\partial}_k\varepsilon_h^{u^\star}-\varepsilon_h^{\widehat{u}}\rangle_{\partial{\mathcal{T}_h}},
		\end{alignat*}
		which leads to
		\begin{align*}
		\|\varepsilon_h^{u^\star}\|^2_{\mathcal{T}_h}&\le
		C h^{\min\{\ell,1\}+1}\|(\mathbb I- \Pi_{\ell}^o)(-\Delta u)\|_{\mathcal{T}_h}|\Psi|_{\min\{\ell,1\}+1} \nonumber\\
		&\quad +Ch^{3/2}\|\bm \Pi^o_k\bm{q}-\bm q\|_{\partial\mathcal{T}_h}
		|\Psi|_2 +Ch\|h_K^{-1/2}(\Pi_{k+1}^\star u -u)\|_{\partial{\mathcal{T}_h}}{|\Psi|_2}\nonumber\\
		&\quad+Ch\left(
		\|\nabla\varepsilon_h^{u^\star}\|_{\mathcal T_h} +\|h_K^{-1/2}(\Pi^{\partial}_k\varepsilon_h^{u^\star}-\varepsilon_h^{\widehat{u}})\|_{\partial\mathcal{T}_h}
		\right)( |\bm\Phi|_1+|\Psi|_2).
		\end{align*}
		Using the elliptic regularity inequality \eqref{regular} and the first inequality of  \Cref{ener_q_ustar}, we finally obtain the wanted result. 
		%	\begin{align*}
		%\|\varepsilon_h^{u^\star}\|_{\mathcal{T}_h} +  \|\varepsilon_h^{u}\|_{\mathcal{T}_h}&\le C h^{1+\min\{\ell,1\}}\|(\mathbb I- \Pi_{\ell}^o)(-\Delta u)\|_{\mathcal{T}_h}\nonumber\\
		%&\quad+C(h^{3/2}\|\bm \Pi^o_k\bm{q}-\bm q\|_{\partial\mathcal{T}_h}+
		%h\|h_K^{-1/2}(\Pi_{k+1}^\star u -u)\|_{\partial{\mathcal{T}_h}}).
	%	\end{align*}
	%	{This completes the proof.  }
	\end{proof}
	
	\subsection*{Step 4: Estimate for $u_h$}
		\begin{lemma} \label{lemma_u}
		We have that
		$\|\varepsilon_h^{u} \|_{\mathcal T_h}  \le \|\varepsilon_h^{u^\star} \|_{\mathcal T_h}.
		$
		\end{lemma}
		Combining this result and the one in the previous step gives the estimate in the approximation for $u$ in \Cref{error_dual2}. {To complete the proof of  \Cref{error_dual2}, it only remains to prove the above lemma.}
	\begin{proof}
		Since $u_h^\star =\mathfrak{p}_h^{k+1}( u_h,\widehat{u}_h)$, $\Pi_{k+1}^\star u=\mathfrak{p}_h^{k+1}(\Pi^{o}_{\ell} u,\Pi^{\partial}_k u)$, and the operator $\mathfrak{p}_h^{k+1}$ is linear, we have that
		$
		\varepsilon_h^{u^\star} =\mathfrak{p}_h^{k+1}( \varepsilon_h^u,\varepsilon_h^{\widehat{u}})$. Proceeding as in the proof of \Cref{lemma1:HHO_proj}, it can be shown that {$ \varepsilon_h^u \in  [\mathcal P_{\ell}^{k+1}(K)]^{\perp}$}.  Then, by equation \eqref{pp2}, the wanted inequality follows.
		{This completes the proof.  } \end{proof}

	\section*{Acknowledgements}
	G.~Chen was supported by supported by National Natural Science Foundation of China (NSFC) grant 11801063, 
		and the Fundamental Research Funds for the Central Universities grant YJ202030. B.~Cockburn was 
	partially supported by National Science Foundation grant DMS-1912646. J.~Singler and Y.~Zhang were supported in part by National Science Foundation grant DMS-1217122. J.~Singler and Y.~Zhang thank the IMA for funding research visits, during which some of this work was completed.
		
	\bibliographystyle{plain}
	\bibliography{yangwen_ref_papers}

\end{document}